\documentclass[11pt,english]{amsart}
\usepackage[T1]{fontenc}
\usepackage[latin1]{inputenc}
\usepackage{amstext}
\usepackage{amsthm}
\usepackage{amssymb}

\makeatletter
\numberwithin{equation}{section}
\numberwithin{figure}{section}
\theoremstyle{plain}
\newtheorem{thm}{\protect\theoremname}
\theoremstyle{definition}
\newtheorem{defn}[thm]{\protect\definitionname}
\theoremstyle{plain}
\newtheorem{prop}[thm]{\protect\propositionname}
\theoremstyle{plain}
\newtheorem{lem}[thm]{\protect\lemmaname}
\theoremstyle{plain}
\newtheorem{cor}[thm]{\protect\corollaryname}
\theoremstyle{remark}
\newtheorem{rem}[thm]{\protect\remarkname}

\usepackage{tikz}
\usepackage{pgfplots}

\usepackage{babel}
\providecommand{\corollaryname}{Corollary}
\providecommand{\definitionname}{Definition}
\providecommand{\lemmaname}{Lemma}
\providecommand{\propositionname}{Proposition}
\providecommand{\remarkname}{Remark}
\providecommand{\theoremname}{Theorem}

\usepackage{babel}
\providecommand{\corollaryname}{Corollary}
\providecommand{\definitionname}{Definition}
\providecommand{\lemmaname}{Lemma}
\providecommand{\propositionname}{Proposition}
\providecommand{\remarkname}{Remark}
\providecommand{\theoremname}{Theorem}

\newcommand{\Ca}{\mathrm{Ca}}
\newcommand{\He}{\mathrm{He}}
\newcommand{\strict}[1]{\mkern 1.5mu\overline{\mkern-1.5mu{#1}\mkern-1.5mu}\mkern 1.5mu}

\makeatother

\usepackage{babel}
\providecommand{\corollaryname}{Corollary}
\providecommand{\definitionname}{Definition}
\providecommand{\lemmaname}{Lemma}
\providecommand{\propositionname}{Proposition}
\providecommand{\remarkname}{Remark}
\providecommand{\theoremname}{Theorem}

\begin{document}
\addtolength{\textwidth}{0mm} \addtolength{\hoffset}{-0mm} \addtolength{\textheight}{0mm}
\addtolength{\voffset}{-0mm}


\global\long\def\AA{\mathbb{A}}%
\global\long\def\CC{\mathbb{C}}%
\global\long\def\BB{\mathbb{B}}%
\global\long\def\PP{\mathbb{P}}%
\global\long\def\QQ{\mathbb{Q}}%
\global\long\def\RR{\mathbb{R}}%
\global\long\def\FF{\mathbb{F}}%
\global\long\def\DD{\mathbb{D}}%
\global\long\def\NN{\mathbb{N}}%
\global\long\def\ZZ{\mathbb{Z}}%
\global\long\def\HH{\mathbb{H}}%
\global\long\def\Gal{{\rm Gal}}%
\global\long\def\OO{\mathcal{O}}%
\global\long\def\pP{\mathfrak{p}}%
\global\long\def\pPP{\mathfrak{P}}%
\global\long\def\qQ{\mathfrak{q}}%
\global\long\def\mm{\mathcal{M}}%
\global\long\def\aaa{\mathfrak{a}}%
\global\long\def\a{\alpha}%
\global\long\def\b{\beta}%
\global\long\def\d{\delta}%
\global\long\def\D{\Delta}%
\global\long\def\L{\Lambda}%
\global\long\def\g{\gamma}%
\global\long\def\G{\Gamma}%
\global\long\def\d{\delta}%
\global\long\def\D{\Delta}%
\global\long\def\e{\varepsilon}%
\global\long\def\k{\kappa}%
\global\long\def\l{\lambda}%
\global\long\def\m{\mu}%
\global\long\def\o{\omega}%
\global\long\def\p{\pi}%
\global\long\def\P{\Pi}%
\global\long\def\s{\sigma}%
\global\long\def\S{\Sigma}%
\global\long\def\t{\theta}%
\global\long\def\T{\Theta}%
\global\long\def\f{\varphi}%
\global\long\def\deg{{\rm deg}}%
\global\long\def\det{{\rm det}}%
\global\long\def\Dem{Proof: }%
\global\long\def\ker{{\rm Ker\,}}%
\global\long\def\im{{\rm Im\,}}%
\global\long\def\rk{{\rm rk\,}}%
\global\long\def\car{{\rm car}}%
\global\long\def\fix{{\rm Fix( }}%
\global\long\def\card{{\rm Card\  }}%
\global\long\def\codim{{\rm codim\,}}%
\global\long\def\coker{{\rm Coker\,}}%
\global\long\def\mod{{\rm mod }}%
\global\long\def\pgcd{{\rm pgcd}}%
\global\long\def\ppcm{{\rm ppcm}}%
\global\long\def\la{\langle}%
\global\long\def\ra{\rangle}%
\global\long\def\Alb{{\rm Alb(}}%
\global\long\def\Jac{{\rm Jac(}}%
\global\long\def\Disc{{\rm Disc(}}%
\global\long\def\Tr{{\rm Tr(}}%
\global\long\def\NS{{\rm NS(}}%
\global\long\def\Pic{{\rm Pic(}}%
\global\long\def\Pr{{\rm Pr}}%
\global\long\def\Km{{\rm Km}}%
\global\long\def\rk{{\rm rk(}}%
\global\long\def\Hom{{\rm Hom(}}%
\global\long\def\End{{\rm End(}}%
\global\long\def\aut{{\rm Aut}}%
\global\long\def\SSm{{\rm S}}%
\global\long\def\psl{{\rm PSL}}%
\global\long\def\cu{{\rm (-2)}}%
\global\long\def\Aut{\mathrm{Aut}}%
\global\long\def\cL{\mathcal{L}}%
\global\long\def\cP{\mathcal{P}}%
\global\long\def\cQ{\mathcal{Q}}%
\global\long\def\cS{\mathcal{S}}%
\global\long\def\cC{\mathcal{C}}%
\global\long\def\DRK#1{{\color{blue}#1\/{\color{black}}}}%
\global\long\def\calK{\mathcal{K}}%

\date{\today}
\title[A special configuration of 12 conics]{A special configuration of 12 conics and generalized Kummer surfaces}
\author{David Kohel, Xavier Roulleau, Alessandra Sarti}
\begin{abstract}
A generalized Kummer surface $X$ obtained as the quotient of an abelian
surface by a symplectic automorphism of order 3 contains a $9{\bf A}_{2}$-configuration
of $(-2)$-curves. Such a configuration plays the role of the $16{\bf A}_{1}$-configurations
for usual Kummer surfaces. In this paper we construct $9$ other such
$9{\bf A}_{2}$-configurations on the generalized Kummer surface associated
to the double cover of the plane branched over the sextic dual curve
of a cubic curve. The new $9{\bf A}_{2}$-configurations are obtained
by taking the pullback of a certain configuration of $12$ conics
which are in special position with respect to the branch curve, plus
some singular quartic curves. We then construct some automorphisms
of the K3 surface sending one configuration to another. We also give
various models of $X$ and of the generic fiber of its natural elliptic
pencil. 
\end{abstract}

\maketitle
\subjclass[2000]{Primary: 14J28}

\section{Introduction}

A Kummer surface $\Km(A)$ is the minimal desingularization of the
quotient of an abelian surface $A$ by the involution $[-1]$. It
is a K3 surface containing $16$ disjoint $\cu$-curves, which lie
over the $16$ singularities of $A/\langle[-1]\rangle$. We call such
set of curves a $16{\bf A}_{1}$-configuration. A well-known result
of Nikulin \cite{Nikulin} gives the converse: if a K3 surface contains
a $16{\bf A}_{1}$-configuration, then it is the Kummer surface of
an abelian surface $A$, such that the $16$ $\cu$-curves lie over
the singularities of $A/\langle[-1]\rangle$.

Shioda \cite{ShiodaRemarks} then asked the following question: if
two abelian surfaces $A$ and $B$ satisfy $\Km(A)\simeq\Km(B)$,
is it true that $A\simeq B$ ? Gritsenko and Hulek \cite{GriH} gave
a negative answer to that question in general. In \cite{RS1}, \cite{RS2},
we studied and constructed examples of two $16{\bf A}_{1}$-configurations
on the same Kummer surface such that their associated abelian surfaces
are not isomorphic.

Kummer surfaces have natural generalizations to quotients of an abelian
surface $A$ by other symplectic groups $G\subseteq\Aut(A)$. If $G\cong\ZZ/3\ZZ$,
then the quotient surface $A/G$ has $9$ cusp singularities, in bijection
with the fixed points of $G$. Its minimal desingularization, denoted
by $\Km_{3}(A)$, is a K3 surface which contains $9$ disjoint $\mathbf{A}_{2}$-configurations,
i.e.\/ pairs $(C,C')$ of $\cu$-curves such that $CC'=1$. It is
then natural to ask if an isomorphism $\Km_{3}(A)\simeq\Km_{3}(B)$
between two generalized Kummer surfaces implies that $A$ and $B$
are isomorphic.

With this question in mind, in the present paper we construct geometrically
several $9\mathbf{A}_{2}$-configurations on some generalized Kummer
surfaces previously studied by Birkenhake and Lange~\cite{BL}. Their
construction is as follows. For $\l$ generic, the dual of a cubic
curve $E_{\l}:x^{3}+y^{3}+z^{3}-3\l xyz=0$ is a sextic curve $C_{\l}$
with a set $\mathcal{P}_{9}$ of $9{\bf A}_{2}$ singularities corresponding
to the nine inflection points on $E_{\l}$. The minimal desingularization
$X_{\l}$ of the double cover of $\PP^{2}$ branched over $C_{\l}$
is a generalized Kummer surface with a natural $9\mathbf{A}_{2}$-configuration
$\mathcal{A}_{0}$. The surface $X_{\l}$ has a natural elliptic fibration
$\varphi:X_{\l}\to\PP^{1}$ for which the $18$ $\cu$-curves in the
$9\mathbf{A}_{2}$-configuration are sections, and the reduced strict
transform of $C_{\l}$ is a fiber. 

In order to find other $\cu$-curves on $X_{\l}$ we study the set
$\mathcal{C}_{12}$ of conics that contain at least $6$ points in
$\mathcal{P}_{9}$. One has 
\begin{thm}
The set $\mathcal{C}_{12}$ has cardinality $12$. Each conic in $\mathcal{C}_{12}$
contains exactly $6$ points in $\mathcal{P}_{9}$ and through each
point in $\mathcal{P}_{9}$ there are $8$ conics. The sets $(\mathcal{P}_{9},\mathcal{C}_{12}$)
form therefore a $(9_{8},12_{6})$-configuration. 
\end{thm}

The configuration $(\cP_{9},\cC_{12})$ has interesting symmetries,
e.g.\/ there are $8$ conics among the $12$ passing through a fixed
point $q$ in $\mathcal{P}_{9}$ and the $8$ points in $\cP_{9}\setminus\{q\}$,
which form a $8_{5}$ point-conic configuration. The freeness of the
arrangement of curves $\cC_{12}$ is studied in \cite{PS}, where
we learned that this configuration has been also independently discovered
in \cite{DLPU}.

The irreducible components of the curves in the K3 surface $X_{\l}$
above the $12$ conics are $24$ $\cu$-curves. This set of $\cu$-curves
contains nine $8{\bf A}_{2}$ sub-configurations which are the strict
transform of the nine $8_{5}$ sub-configurations of conics. Using
the pullback to $X_{\l}$ of some $9$ special (singular) quartic
curves, we are able to complete each of these $8{\bf A}_{2}$-configurations
into a new $9{\bf A}_{2}$-configuration $\mathcal{A}_{k},\,(k\in\{1,\dots,9\})$.

According to \cite{Barth}, to a $9{\bf {A}_{2}}$-configuration corresponds
an Abelian surface $A$ and an order $3$ symplectic group $G$ such
that $X_{\l}\simeq\Km_{3}(A)$ and the $9{\bf {A}_{2}}$-configuration
is the exceptional divisor of the minimal desingularisation $\Km_{3}(A)\to A/G$
(this is the analog of Nikulin's result for $16{\bf {A}_{1}}$-configurations).

A Kummer structure on K3 surface $X$ is an isomorphism class of Abelian
surfaces $A$ such that $X\simeq\Km(A)$. Kummer structures are in
one-to-one correspondence with the orbits under $\aut(X)$ of Nikulin's
$16{\bf {A}_{1}}$-configurations (see e.g.~\cite[Proposition 21]{RS1}).

Similarly, one can define a generalized Kummer structure on a K3 surface
$X$ as an isomorphism class of pairs $(A,G$) of abelian surfaces
$A$ and order $3$ symplectic group $G$ such that $X\simeq\Km_{3}(A)$.
We show that there is again a one-to-one correspondence between generalized
Kummer structures and the orbits of the $9{\bf {A}_{2}}$-configurations.
Using the $9{\bf {A}_{2}}$-configurations we constructed, we obtain: 
\begin{thm}
The $9{\bf {A}_{2}}$-configurations $\mathcal{A}_{1},\dots,\mathcal{A}_{9}$
we obtained on the K3 surface $X_{\l}$ are contained in the $\aut(X_{\lambda})$-orbit
of $\mathcal{A}_{0}$. 
\end{thm}

For the proof we construct automorphisms sending one configuration
to another. Some of these automorphisms are obtained by using translations
by torsion sections of the natural elliptic fibration $\varphi:X_{\l}\to\PP^{1}$,
and some other by using the Torelli Theorem for K3 surfaces.

We then continue our study of the surface $X_{\l}$ by obtaining various
models in projective space, in particular as a degree $8$ non-complete
intersection in $\PP^{5}$. We construct a model of the generic fiber
$E_{K3}$ of the fibration $\varphi$. 
\begin{thm}
The fibration $\varphi:X_{\l}\to\PP^{1}$ has $8$ singular fibers
of type $\tilde{{\bf A}}_{2}$; the $24$ $\cu$-curves above the
$12$ conics in $\mathcal{C}_{12}$ are contained in these fibers.
A Hessian model of the generic fiber of $\varphi$ is 
\[
E_{K3}:x^{3}+y^{3}+z^{3}+\frac{\l^{3}(t^{2}+3)-4t^{2}}{\l^{2}(t^{2}-1)}xyz=0.
\]
\end{thm}

We also get a Weierstrass model of $E_{K3}$. It turns out that the
Mordell-Weil group of the fibration $\varphi$ has rank $1$ and torsion
$(\ZZ/3\ZZ)^{2}$; we compute its generators. Using the translation
maps constructed from the model $E_{K3}$, we can acquire other $9\mathbf{A}_{2}$-configurations
from the previously known one. We also obtain another construction
of the K3 surface $X_{\l}$ as a double plane: 
\begin{thm}
The surface $X_{\l}$ is the minimal desingularization of the double
cover of $\PP^{2}$ branched over the sextic curve which is the union
of the elliptic curves $E_{\l}$ and its Hessian 
\[
\He(\l):\,\,\,\,x^{3}+y^{3}+z^{3}+\frac{(\l^{3}-4)}{\l^{2}}xyz=0,
\]
 The strict transform on $X_{\l}$ of the $12$ lines of the Hesse
arrangement in $\PP^{2}$ are the $24$ $\cu$-curves above the $12$
conics. 
\end{thm}

\textbf{Acknowledgements} The authors wish to thank Cédric Bonnafé,
Igor Dolgachev, Antonio Laface, Ulf Persson, Piotr Pokora, Giancarlo
Urzúa, and also Carlos Rito for sharing his program \verb"LinSys".
Part of the computations were done using Magma software \cite{Magma}.
The second author thanks the Max-Planck Institute for Mathematics
of Bonn for its hospitality and support. The third author is partially
supported by the ANR project No. ANR-20-CE40-0026-01 (SMAGP).

\section{Preliminaries}

\subsection{Notations and conventions\label{subsec:Notations-and-conventions}}

Let $\eta:Y\to Z$ be a dominant map between two surfaces and let
$C\hookrightarrow Z$ be a curve. In this paper, the reduced pullback
of $C$ minus the irreducible components contracted by $\eta$ is
called the strict transform of $C$ on $Y$. 
\begin{defn}
We say that two $\cu$-curves $E,E'$ on a K3 surface form an $\mathbf{A}_{2}$-configuration
if their intersection matrix is: $\left(\begin{array}{@{}c@{\,}c@{}}
-2 & \ 1\\
\ 1 & -2
\end{array}\right)\cdot$ We say that the $\cu$-curves $E_{1},E_{1}',\dots,E_{n},E_{n}'$
form an $n{\bf A}_{2}$-configuration if the curves $E_{j},E_{j}',\,j\in\{1,\dots,n\}$
form $n$ disjoint ${\bf A}_{2}$-configurations. 
\end{defn}

We recall (see \cite{Dolgachev}) that a $(v_{r},b_{k})$-configuration
is the data of two sets $\mathcal{P}_{v},\mathcal{L}_{b}$ of respective
cardinality $v$ and $b$, plus a subset $I\subset\mathcal{P}_{v}\times\mathcal{L}_{b}$,
such that for each element $p$ in $\mathcal{P}_{v}$, there are $r$
elements $l$ in $\mathcal{L}_{b}$ such that $(p,l)\in I$, and for
each element $l$ in $\mathcal{L}_{b}$, there are $k$ elements $p$
in $\mathcal{P}_{v}$ such that $(p,l)\in I$. One has $vr=bk$. When
$v=b$ and $r=k$, we call such a configuration a $v_{r}$-configuration.

\subsection{Generalized Kummer structures\label{subsec:Generalized-Kummer-structures}}

Let $X$ be a K3 surface. Suppose that $\mathcal{C}=A_{1},A_{1}',\dots,A_{9},A_{9}'$
is a $9{\bf A}_{2}$-configuration on $X$. By the results of Barth
\cite{Barth}, there exist coefficients $a_{j},a_{j}'\in\{1,2\}$
and a divisor $D\in\NS X)$ such $\sum_{j=1}^{9}(a_{j}A_{j}+a_{J}'A_{j}')=3D$.
Let $\tilde{X}\to X$ be the blow-up at the $9$ intersection points
$A_{k}A_{k}'$ ($k=1,...,9$). The strict transform of $\mathcal{C}$
on $\tilde{X}$ is the union of $18$ disjoint $(-3)$-curves. The
quoted results of Barth gives that there exists a unique triple cover
map $\tilde{B}\to\tilde{X}$ (for the theory of cyclic covers, see
\cite[Chapter 1, Section 17]{BHPVdV}) branched over the $18$ $(-3)$-curves;
the reduced pull-back of these curves are $(-1)$-curves and the pull-back
of the exceptional curves on $\tilde{X}$ are $9$ $(-3)$-curves.
The minimal model of $\tilde{B}$ is then an abelian surface $B$
with an order $3$ symplectic automorphism group $G$ coming from
the cyclic cover, and the surface $X$ is (isomorphic to) the minimal
desingularization of the quotient surface $B/G$. 
\begin{defn}
A generalized Kummer structure (of order $3$) on a K3 surface $X$
is an isomorphism class of pairs $(A,G)$ of abelian surfaces equipped
with an order $3$ symplectic automorphism subgroup $G\subset\aut(A)$,
such that $X\simeq\Km_{3}(A)$, where $\Km_{3}(A)$ is the minimal
desingularization of $A/G$. 
\end{defn}

\begin{prop}
There is a one-to-one correspondence between Kummer structures on
$X$ and $\aut(X)$-orbits of $9{\bf A}_{2}$-configurations. 
\end{prop}

\begin{proof}
Let $\mu:X\to X$ be an automorphism of $X$ sending the configuration
$\mathcal{C}$ to the configuration $\mathcal{C}'$. Let $B,B'$ be
the abelian surfaces and let $G\subset\aut(B),\,G'\subset\aut(B')$
be the order $3$ automorphism groups such that $X$ is the minimal
desingularization of $B/G$ and $B'/G'$, and the exceptional curves
of the minimal desingularization are respectively in $\mathcal{C}$,
$\mathcal{C}'$. Let us prove that there exists an isomorphism $\tau$
between the abelian surfaces $B,B'$ such that $G'=\tau G\tau^{-1}$.
\\
As above let $\tilde{X}$, $\tilde{X}'$ be the blow-up of $X$ at
the $9$ singular points of $\mathcal{C}$ and $\mathcal{C}'$ respectively.
The automorphism $\mu$ extends to an isomorphism $\tilde{\mu}:\tilde{X}\to\tilde{X}'$.
The map $\tilde{B}\to\tilde{X}\stackrel{\tilde{\mu}}{\to}\tilde{X}'$
is branched with order $3$ over the strict transform of $\mathcal{C}'$
in $\tilde{X}'$. By uniqueness of the triple cover, there is an isomorphism
$\tilde{\tau}:\tilde{B}\to\tilde{B}',$ (which gives an isomorphism
to $\tau:B\to B'$ between the minimal models). Moreover, the order
$3$ automorphisms group $\tilde{G}$ of transformations of the triple
cover $\tilde{B}\to\tilde{X}$ is sent by transport of structures
to the order $3$ automorphism group $\tilde{G}'$ of transformations
of the triple cover $\tilde{B}'\to\tilde{X}'$, i.e. $\tilde{G}'=\tilde{\tau}\tilde{G}\tilde{\tau}^{-1}$.
This property is preserved when taking the minimal models: $G'=\tau G\tau^{-1}$.

For the converse, let $(B,G)$ and $(B',G')$ be equipped abelian
surfaces such that $\Km_{3}(B)=X=\Km_{3}(B')$ and there is an isomorphism
$\tau:B\to B'$ with $G=\tau^{-1}G'\tau$. The map $B\to B'\to B'/G'$
is $3$ to $1$ and $G$-invariant, thus it induces an isomorphism
$B/G\simeq B'/G'$. That isomorphism extends to the desingularization:
$X=\Km_{3}(B)\simeq\Km_{3}(B')=X$ and sends the first $9{\bf A}_{2}$-configuration
to the second. 
\end{proof}

\subsection{The Néron-Severi and transcendental lattices\label{subsec:The-N=00003D0000E9ron-Severi-group}}

In this Section, which can be skipped in a first reading, we compute
the Néron-Severi lattice and the transcendent lattice of the generalized
Kummer surface $X$, the results will be used in Section \ref{sec:Nine-new-}.

\subsubsection{The Néron-Severi lattice}

Let $A$ be an abelian surface with a symplectic action of a group
$G:=\ZZ/3\ZZ$. The quotient $A/G$ has $9{\bf A}_{2}$ singularities
and the minimal resolution is a generalized Kummer surface $X:=\Km_{3}(A)$
which carries a $9{\bf A}_{2}$-configuration $A_{1},A_{1}',\ldots,A_{9},A_{9}'$
of $\cu$-curves.

Observe that the abelian surface has Picard number at least $3$,
see \cite[Proposition on p. 10]{barth2} and the K3 surface $X$ has
generically Picard number $\rho(X)$ equal to $19$. Let $\calK_{3}$
denote the minimal primitive (rank $18$) sub-lattice of the K3 lattice
$H^{2}(X,\ZZ)$ containing the $9$ configurations ${\bf A}_{2}$.
The lattice $\calK_{3}$ is described as follows. By \cite[Proof of Proposition 1.3]{bertin},
it is generated by the classes $A_{1},A_{1}',\ldots,A_{9},A_{9}'$
and the following three classes 
\[
\begin{array}{c}
v_{1}=\frac{1}{3}\sum_{i=1}^{9}(A_{i}-A_{i}')\\
v_{2}=\frac{1}{3}((A_{2}-A{}_{2}')+2(A_{3}-A_{3}')+A_{6}-A_{6}'+2(A_{7}-A_{7}')+A_{8}-A_{8}'+2(A_{9}-A_{9}'))\\
v_{3}=\frac{1}{3}(A_{4}-A_{4}'+2(A_{5}-A_{5}')+A_{6}-A_{6}'+2(A_{7}-A_{7}')+2(A_{8}-A_{8}')+A_{9}-A_{9}').
\end{array}
\]
Then the discriminant group $\calK_{3}^{\vee}/\calK_{3}$ is generated
by the classes 
\[
\begin{array}{c}
w_{1}=\frac{1}{3}(A_{5}-A_{5}'+A_{7}-A_{7}'+A_{8}-A_{8}')\\
w_{2}=\frac{1}{3}(2(A_{4}-A_{4}')+A_{6}-A_{6}'+2(A_{7}-A_{7}')+A_{8}-A_{8}')\\
w_{3}=\frac{1}{3}(A_{3}-A_{3}'+A_{5}-A_{5}'+A_{6}-A_{6}')
\end{array}
\]
with intersection matrix (the coefficients are in $\QQ/2\ZZ$): 
\[
\left(\begin{array}{ccc}
0 & 0 & -\frac{2}{3}\\
0 & -\frac{2}{3} & -\frac{2}{3}\\
-\frac{2}{3} & -\frac{2}{3} & 0
\end{array}\right).
\]
Assume $\rho(X)=19$ (the minimal possible) and that $D_{2}$, a generator
of $\calK_{3}^{\perp}\subset\NS X)$, has square $D_{2}^{2}=2$. We
state here the following Lemma for later use in Sections \ref{subsec:NaturalEllFib}
\ref{subsec:stabilizerGroupFirstConf}:
\begin{lem}
\label{LEMMA:The-NS-latt-Discri54}The Néron-Severi lattice is $\NS X)=\ZZ D_{2}\oplus\calK_{3}$
and its discriminant group has order $54$. 
\end{lem}

\begin{proof}
 The lattice $\mathcal{K}_{3}$ is $3$-elementary and $D_{2}^{2}=2$,
thus the result.
\end{proof}

\subsubsection{The transcendental lattice}

Since $\NS X)=\ZZ D_{2}\oplus\mathcal{K}_{3}$, we know that the discriminant
group is 
\[
\tfrac{\NS X)^{\vee}}{\NS X)}=\tfrac{(\ZZ D_{2})^{\vee}}{\ZZ D_{2}}\oplus\tfrac{\mathcal{K}_{3}^{\vee}}{\mathcal{K}_{3}}.
\]
The intersection matrix of the generators $w_{1},w_{2},w_{3}-w_{1}-w_{2}+\frac{1}{2}D_{2}$
is

\begin{eqnarray*}
\mathcal{N}:=\left(\begin{array}{ccc}
0&0&-\frac{2}{3}\\
0&-\frac{2}{3}&0\\
-\frac{2}{3}&0&\frac{1}{2}\\
\end{array}
\right).
\end{eqnarray*}

Recall that the quadratic form on the discriminant group has values
in $\QQ/2\ZZ$ and the corresponding bilinear form has values in $\QQ/\ZZ$,
see e.g.~\cite[Section 2]{morrison}. The transcendental lattice
$T_{X}$ of $X$ has rank $3$ signature $(2,1)$, and since $H^{2}(X,\ZZ)$
is unimodular $\det\,T_{X}=-2\cdot3^{3}$ with discriminant form the
same as the discriminant form of $\NS X)$ scaled by $-1$. As in
\cite[Definition 2.1]{sarti_tran}, we define
\begin{defn}
We call the discriminant $d$ of an indefinite rank $3$ lattice {\it small}
if $4\cdot d$ is not divisible by $k^{3}$ for any non square natural
number $k$ congruent to $0$ or $1$ modulo $4$.
\end{defn}

The lattice $T_{X}$ is small, thus by \cite[Theorem 21, p. 395]{conway}
and \cite[Corollary 1.9.4]{conway} it is uniquely determined by its
signature and its discriminant form. Consider now the rank three lattice
$T$ of signature $(2,1)$ with Gram matrix\begin{eqnarray*}
\left(\begin{array}{ccc}
0&3&0\\
3&6&-3\\
0&-3&6\\
\end{array}
\right).
\end{eqnarray*}It has determinant $-2\cdot3^{3}$ and if one calls $u_{1}$, $u_{2}$,
$u_{3}$ its generators, one computes that the discriminant group
is generated by 
\[
\tfrac{2u_{1}}{3},\,\,\tfrac{u_{3}}{3},\,\,\tfrac{3u_{1}+2u_{2}+u_{3}}{6}
\]
and the discriminant form is the same as $-\mathcal{N}$ so that by
unicity:
\begin{prop}
The transcendental lattice $T_{X}$ is isometric to $T$. 
\end{prop}

\subsection{The Hesse, dual Hesse, and related configurations\label{subsec:The-Hesse,-dual}}

Let us recall the construction and properties of the Hesse configuration
before introducing a configuration with analogous properties in the
next section.

Let $E\hookrightarrow\PP^{2}$ be a smooth cubic curve and $\mathcal{T}_{9}$
its set of $9$ inflection points. Fixing an inflection point as the
group identity, the set $\mathcal{T}_{9}$ is the $3$-torsion subgroup
$E[3]$. The \textit{Hesse configuration} is defined to be the pair
$(\mathcal{T}_{9},\cL_{12})$, where $\cL_{12}$ is the set of $12$
lines through pairs of points in $\mathcal{T}_{9}$. Since each line
meets $\mathcal{T}_{9}$ at $3$ points and each point of $\mathcal{T}_{9}$
is contained in $4$ lines, the Hesse configuration is a $(9_{4},12_{3})$-configuration.
Removing one point of the Hesse configuration and its $4$ incident
lines, one obtains a symmetric $8_{3}$-configuration of $8$ points
and $8$ lines.

We note that the construction is not unique to the curve $E$. If
$\cP$ is the set of inflection points of any nonsingular cubic plane
curve $E$, then there exists a pencil of elliptic curves (called
the Hesse pencil) whose set of inflection points is $\cP$.

By taking the $12$ points and $9$ lines in the dual space $(\PP^{2})^{*}$
corresponding to the $12$ lines and $9$ points dual to the Hesse
configuration, one obtains a dual $(12_{3},9_{4})$-configuration
called the \textit{dual Hesse configuration}. As above, removing one
line and the 4 points it meets gives a dual symmetric $8_{3}$-configuration.

As abstract configurations, the above configurations can be realized
from the symmetric $13_{4}$-configuration of points and lines in
the plane $\PP^{2}(\mathbb{F}_{3})$. Removing a fixed point and the
set of $4$ lines which pass through it, gives the $(12_{3},9_{4})$-configuration,
and removing a fixed line and the 4 points it contains gives the $(9_{4},12_{3})$-configuration.
The $(9_{4},12_{3})$-configuration is naturally identified with $\AA^{2}(\FF_{3})=\PP^{2}(\FF_{3})\setminus\PP^{1}(\FF_{3})$
equipped with its system of affine lines, consistent with the identification
$\mathcal{T}_{9}=E[3]\cong\FF_{3}^{2}$ in the Hesse configuration.

\section{\label{sec:Nine-new-}Nine new $9\mathbf{A}_{2}$-configurations}

\subsection{\label{subsec:confOfConics}$(9_{8},12_{6})$ and $8_{5}$-configurations
of conics}

Let us fix $\l\notin\{1,\omega,\omega^{2}\}$, for $\omega$ such
that $\omega^{2}+\omega+1=0$. The dual $C_{\l}$ of the elliptic
curve 
\[
E_{\l}:\,\,x^{3}+y^{3}+z^{3}-3\l xyz=0,
\]
is a $9$-cuspidal sextic curve, (i.e. a sextic curve with $9$ cusps)
and conversely any $9$-cuspidal sextic curve is obtained in that
way. The curve $C_{\l}$ is: 
\[
\begin{array}{c}
C_{\l}:\,\,(x^{6}+y^{6}+z^{6})+2(2\l^{3}-1)(x^{3}y^{3}+x^{3}z^{3}+y^{3}z^{3})\\
\hfill-6\l^{2}xyz(x^{3}+y^{3}+z^{3})-3\l(\l^{3}-4)x^{2}y^{2}z^{2}=0.
\end{array}
\]
The images by the dual map of the $9$ inflection points of $E_{\l}$
are the $9$ cusps of $C_{\l}$. The set $\mathcal{P}_{9}$ of the
$9$ cusps $p_{1},\dots,p_{9}$ is 
\[
\begin{array}{ccc}
p_{1}=(\l:1:1), & p_{4}=(\l:\omega:\omega^{2}), & p_{7}=(\l:\omega^{2}:\omega)\\
p_{2}=(1:\l:1), & p_{5}=(\omega^{2}:\l:\omega), & p_{8}=(\omega:\l:\omega^{2})\\
p_{3}=(1:1:\l), & p_{6}=(\omega:\omega^{2}:\l), & p_{9}=(\omega^{2}:\omega:\l).
\end{array}
\]
When $\l$ varies, the closure of the set of points $p_{j}$ is a
line, denoted by $L_{j}$; we obtain in that way a set $\mathcal{L}_{9}$
of $9$ lines. Dually, the points on $L_{j}$ correspond to the pencil
of lines meeting in the inflection point (corresponding to $p_{j}$)
of the elliptic curve $E_{\l}$. One can check moreover that the line
$L_{j}$ is the tangent line to the cusp $p_{j}$. Let $\mathcal{P}_{12}$
be the intersection points set of the lines in $\mathcal{L}_{9}$. 
\begin{thm}
\label{thm:12Conics9Points} The set $\mathcal{P}_{12}$ has cardinality
$12$. The pair $(\mathcal{P}_{12},\mathcal{L}_{9})$ forms a $(12_{3},9_{4})$-configuration
which is the dual Hesse configuration. The set $\cC_{12}$ of conics
containing $6$ points in $\cP_{9}$ has cardinality $12$; each conic
of $\mathcal{C}_{12}$ is smooth. The pair of sets $(\cP_{9},\cC_{12})$
of points and conics form a $(9_{8},12_{6})$-configuration. \\
 More precisely, the union of the pairwise intersections of the $12$
conics in $\mathcal{C}_{12}$ is $\mathcal{P}_{9}\cup\mathcal{P}_{12}$.
The intersections between the conics are transverse. Moreover, each
point of $\mathcal{P}_{12}$ is contained in exactly two conics. A
conic $C$ in $\mathcal{C}_{12}$ meets $9$ conics of $\mathcal{C}_{12}$
in $4$ points of $\mathcal{P}_{9}$ and the two remaining conics
in $3$ points of $\mathcal{P}_{9}$ and one point of $\mathcal{P}_{12}$. 
\end{thm}

\begin{proof}
By a computer search, the $12$ conics are: 
\[
\begin{array}{c}
C_{1,2,3,4,5,6}:\,\,x^{2}+(\l+1)(\omega xy+\omega^{2}xz+yz)+\omega^{2}y^{2}+\omega z^{2}=0,\hfill\\
C_{1,2,3,7,8,9}:\,\,x^{2}+(\l+1)(\omega^{2}xy+\omega xz+yz)+\omega y^{2}+\omega^{2}z^{2}=0,\hfill\\
C_{1,2,4,5,7,8}:\,\,xy-\l z^{2}=0,\hfill\\
C_{1,2,4,6,8,9}:\,\,x^{2}+(\omega\l+1)(xy+\omega xz+\omega yz)+y^{2}+\omega^{2}z^{2}=0,\hfill\\
C_{1,2,5,6,7,9}:\,\,x^{2}+(\omega^{2}\l+1)(xy+\omega^{2}xz+\omega^{2}yz)+y^{2}+\omega z^{2}=0,\hfill\\
C_{1,3,4,5,8,9}:\,\,x^{2}+(\omega\l+\omega^{2})(xy+yz+\omega xz)+\omega y^{2}+z^{2}=0,\hfill\\
C_{1,3,4,6,7,9}:\,\,-\l y^{2}+xz=0,\hfill\\
C_{1,3,5,6,7,8}:\,\,x^{2}+(\omega^{2}\l+\omega)(xy+yz+\omega^{2}xz)+\omega^{2}y^{2}+z^{2}=0,\hfill\\
C_{2,3,4,5,7,9}:\,\,x^{2}+(\l+\omega^{2})(xy+xz+\omega^{2}yz)+\omega y^{2}+\omega z^{2}=0,\hfill\\
C_{2,3,4,6,7,8}:\,\,x^{2}+(\l+\omega)(xy+xz+\omega yz)+\omega^{2}(y^{2}+z^{2})=0,\hfill\\
C_{2,3,5,6,8,9}:\,\,\l x^{2}-yz=0,\hfill\\
C_{4,5,6,7,8,9}:\,\,x^{2}+(\l+1)(xy+xz+yz)+y^{2}+z^{2}=0,\hfill
\end{array}
\]
where the index $i,j,\dots,n$ of the conic $C_{i,j,\dots,n}$ means
that this conic contains the $6$ points $p_{s},\,s\in\{i,j,\dots,n\}$.
It is easy to see that the points in $\mathcal{P}_{9}$ are in general
position: no line contains $3$ cusps, thus the conics are smooth.
From the data of the conics and the knowledge of the points in $\mathcal{P}_{9}$
they contain, one can check the assertions about the configuration
of the $12$ conics and the $9$ points. If one renumbers the $12$
conics by their order $C_{1},\dots,C_{12}$ from the top to bottom
of the above list, one obtains that the pair of (indexes of) conics
which have an intersection point not in $\mathcal{P}_{9}$ are 
\[
\begin{array}{c}
(1,2),\,(1,12),\,(2,12),\,(3,7),\,(3,11),\,(4,8),\\
(4,9),\,(5,6),\,(5,10),\,(6,10),\,(7,11),\,(8,9),
\end{array}
\]
and correspondingly, the $12$ points are 
\[
\begin{array}{c}
(1:1:1),\,(\omega:\omega^{2}:1),\,(\omega^{2}:\omega:1),\,(1:0:0),\,(0:1:0),\,(\omega^{2}:1:1),\\
(1:\omega^{2}:1),\,(\omega:1:1),\,(1:\omega:1),\,(\omega^{2}:\omega^{2}:1),\,(0:0:1),\,(\omega:\omega:1).
\end{array}
\]
respectively. One can check easily that these $12$ points in $\mathcal{P}_{12}$
are the intersection points of the lines in $\mathcal{L}_{9}$, which
lines form the dual Hesse arrangement (see Section \ref{subsec:The-Hesse,-dual}).
By Bézout's Theorem, the intersections between the conics are transverse. 
\end{proof}
Let $q\in\mathcal{P}_{9}$ and define $\mathcal{P}_{q}=\mathcal{P}_{9}\setminus\{q\}$. 
\begin{thm}
\label{thm:8_5Config} The set $\mathcal{C}_{q}$ of conics containing
the point $q$ has cardinality $8$.\\
 The set of points $\mathcal{P}_{q}$ and the set of conics $\mathcal{C}_{q}$
form an $8_{5}$-configuration: each point is on $5$ conics and each
conic contains $5$ of the points in $\mathcal{P}_{q}$. \\
 For each conic $C$ in $\mathcal{C}_{q}$ there exists a unique conic
$C'\in\mathcal{C}_{q}$ such that there is a unique point in the intersection
of $C$ and $C'$ which is not in $\mathcal{P}_{q}$. 
\end{thm}

\begin{proof}
That can be checked directly from the datas in the proof of Proposition
\ref{thm:12Conics9Points}. 
\end{proof}
Let $X_{\l}$ be the minimal desingularization of the double cover
branched over the sextic curve $C_{\l}$ with $9$ cusps. We denote
by $\eta:X_{\l}\to\PP^{2}$ the natural map and we denote by $A_{j},A_{j}'$
the two $\cu$-curves in $X$ above the point $p_{j}$ in $\mathcal{P}_{9}$.
The curves $A_{j},A_{j}'$, $j\in\{1,\dots,9\}$ form a $9\mathbf{A}_{2}$-configuration,
which we call the \textit{natural $9\mathbf{A}_{2}$-configuration}.
We have 
\begin{lem}
\label{lem:2A2config}The strict transform by the map $\eta:X_{\l}\to\PP^{2}$
of a conic $C\in\mathcal{C}_{12}$ is the union of two disjoint $\cu$-curves
$\t_{C},\t_{C}'$. \\
 Let $C,D$ be two conics in $\mathcal{C}_{12}$. Suppose that $C$
and $D$ meet in $4$ points in $\mathcal{P}_{9}$. Then the $\cu$-curves
$\t_{C},\t_{C}',\t_{D},\t_{D}'$ are disjoint. \\
 Suppose that $C$ and $D$ meet in $3$ points in $\mathcal{P}_{9}$.
Then, up to exchanging $\t_{D}$ and $\t_{D}'$, the curves $\t_{C},\t_{D},\t_{C}',\t_{D}'$
form a $2\mathbf{A}_{2}$-configuration. 
\end{lem}

\begin{proof}
Rather than performing a double cover and taking the resolution of
surface singularities, we perform one blow-ups at each cusp $q\in\mathcal{P}_{9}$
of $C_{\l}$, so that the branch locus is smooth and tangent to the
exceptional curve $E$ (see also Figure \ref{subsec:confOfConics}).
On the double cover, the reduced image inverse of the curve $E$ is
the union of two $\cu$-curves on the smooth K3 surface $X_{\l}$.\\
 By that local computation, we see that for $C\in\mathcal{C}_{12}$,
the curves $\t_{C},\t_{C'}$ are disjoint (the strict transform $\strict{C}$
of $C$ under the blow-up map do not meet the branch locus). \\
 Suppose $C$ and $D$ meet in $4$ points in $\mathcal{P}_{9}$.
The intersection being transverse, each strict transforms $\strict{C},\strict{D}$
of the curves $C,D$ under the $3$ blow-ups at each cusps is the
union of two disjoint curves not meeting the branch curve and the
$4$ curves in $\strict{C},\strict{D}$ are disjoint. \\
 If $C$ and $D$ meet in $3$ points in $\mathcal{P}_{9}$ then they
meet transversely at a unique point not in $\mathcal{P}_{9}$. Then
taking the above notations, we have this time $\strict{C}\strict{D}=1$,
so that the last assertion holds. 
\end{proof}
Let $\mathcal{P}_{q}$ and $\mathcal{C}_{q}$ as above. Using Theorem
\ref{thm:8_5Config} and Lemma \ref{lem:2A2config}, we get: 
\begin{cor}
The strict transform by $\eta$ of the $8$ conics in $\mathcal{C}_{q}$
forms an $8\mathbf{A}_{2}$-configuration. 
\end{cor}

For each point $q=p_{j},$ $j\in\{1,\dots,9\}$, we denote by $\mathcal{A}_{j}'$
the corresponding $8\mathbf{A}_{2}$-configuration on $X_{\l}$. In
order to obtain new $9{\bf A}_{2}$-configurations, one needs to find
other $\mathbf{A}_{2}$-configurations. This will be done in the next
section by using singular quartics instead of conics. 
\begin{rem}
Using a computer, we found eight $8{\bf A}_{2}$-configurations in
the set of $32$ $\cu$-curves which is the union of the two $8\mathbf{A}_{2}$-configurations
$\mathcal{A}_{j}'$ and $\{A_{1},A_{1}',\dots,A_{9},A_{9}'\}\setminus\{A_{j},A_{j}'\}$.
However one can compute that the orthogonal complement of $6$ of
them are lattices with no $\cu$-classes, thus one cannot complete
these $6$ configurations into $9\mathbf{A}_{2}$-configurations.
The $32$ $\cu$-curves can be realized as lines in a projective model
of $X_{\l}$, see Proposition \ref{prop:degree8Model}. 
\end{rem}

\subsection{\label{subsec:-9new-9A2-configurations}Nine new $9\mathbf{A}_{2}$-configurations }

Let $p_{j}\in\mathcal{P}_{9}$ be one of the $9$ cusp singularities
of the sextic $C_{\l}$. 
\begin{thm}
There exists a unique quartic curve $Q_{j}$ passing through the $9$
points in $\mathcal{P}_{9}$, with a unique singularity at the point
$p_{j}$. The singularity has multiplicity $3$, it is of type ${\bf D}_{5}$:
it has two tangents, one branch is smooth while the other branch is
a cuspidal singularity. The tangent at the cusp of $Q_{j}$ is also
the tangent to the cuspidal singularity of the sextic $C_{\l}$ at
$p_{j}$. \\
 The curve $Q_{j}$ has geometric genus $0$. Its strict transform
denoted by $\eta$ on $X_{\l}$ is the union of two $\cu$-curves
$\g_{j},\g_{j}'$ which form an $\mathbf{A}_{2}$-configuration. The
curves $\g_{j},\g_{j}'$ and the $16$ curves in $\mathcal{A}_{j}'$
form a $9\mathbf{A}_{2}$-configuration $\mathcal{A}_{j}$. 
\end{thm}

\begin{proof}
We give in the Appendix the equations of the $9$ curves $Q_{j},$
$j\in\{1,\dots,9\}$. These curves have been constructed using the
\verb"LinSys" program by C. Rito which enables to find curves of
given degree with prescribed singularities and given tangencies at
a set of points in the plane. Conversely, one can check that the singularity
of $Q_{j}$ at $p_{j}$ has multiplicity $3$, is resolved by one
blow-up, with the exceptional divisor meeting the strict transform
in two points, one of multiplicity $2$.\\
 The curve $Q_{j}$ has genus $0$, (see e.g.~\cite[Chapter 4, Section 2]{GH}).
By Bézout's Theorem, the intersections of the quartic $Q_{j}$ with
the $8$ conics in $\mathcal{C}_{12}$ that contain $p_{j}$ are transverse,
so that the curves in $\mathcal{A}_{j}'$ are disjoint from $\g_{j},\g_{j}'$
and we thus get a $9\mathbf{A}_{2}$-configuration. See Figure \ref{subsec:confOfConics}
for the behavior of the quartic curve $Q_{j}$ under the double cover. 
\end{proof}
\begin{figure}[h]
\label{figniceDiag} \caption{Behavior of the quartic $Q_{j}$ under the double cover}

\centering{}
\begin{tikzpicture}[scale=0.9]

\draw (1.8,-0.3) node {\small blow-up};
\draw (4.3,-0.3) node {\small 2:1 cover};
\draw (8.2,0) node {$X_\lambda$};
\draw (-1,0) node {$\mathbb{P}^2$};

\draw [color=blue] [domain=0:1] plot (\x, \x^3);
\draw [color=blue] [domain=0:1] plot (\x, -\x^3);
\draw [very thick] [domain=-0.82:0] plot (\x, 2*\x^3);
\draw [very thick] [domain=-0.85:0] plot (\x, -2*\x^3);
\draw [color=blue] (1,1) arc (0:180:0.5) ;
\draw [color=blue] (0,-1) -- (0,1);
\draw (0.75,-1) node {$Q_j$};
\draw (-0.48,-1) node {$C_\lambda$};

\draw [<-] (1.5,0) -- (2.2,0);
\draw [color=blue] [domain=0:1.05] plot (\x/2+3, sqrt abs \x);
\draw [color=blue] [domain=0:1.8] plot (\x/2+3, -sqrt abs \x);
\draw [very thick] [domain=-1.8:0] plot (\x/2+3, sqrt abs \x);
\draw [very thick] [domain=-1.8:0] plot (\x/2+3, -sqrt abs \x);
\draw (3,-1.3) -- (3,1.9);
\draw [color=blue] (3.5,1) arc (-35:145:0.5) ;
\draw (3.25,-1.2) node {-1};
\draw (2.75,-1.2) node {$E$};

\draw [->] (3.9,0) -- (4.6,0);

\draw [very thick] (0.8+4.7,-2.2+3.5) -- (0.8+4.7,-4.5+3.5);
\draw (0.2+4.7,-3.8+3.5) -- (2.8+4.7,-2.5+3.5);
\draw (0.2+4.7,-3.2+3.5) -- (2.8+4.7,-4.5+3.5);
\draw [color=blue] (0.4+4.7,-3.1+3.5) -- (1.6+4.7,-4.3+3.5);
\draw [color=blue] (0.4+4.7,-3.9+3.5) -- (1.6+4.7,-2.7+3.5);
\draw [color=blue] (1.6+4.7,-2.7+3.5) arc (135:18:0.65) ;
\draw [color=blue] (1.6+4.7,-4.3+3.5) arc (360-135:360-20:0.65) ;
\draw (2.9+4.7,-2.4+3.5) node {$-2$};
\draw (2.9+4.7,-4.6+3.5) node {$-2$};

\draw (2.7+4.7,-3.9+3.5) node {$\gamma_j$};
\draw (2.7+4.7,-3.2+3.5) node {$\gamma_j'$};
\end{tikzpicture}
\end{figure}
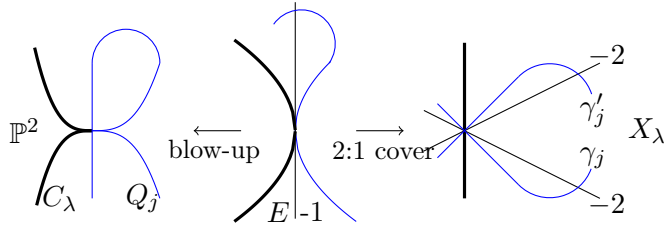

\begin{rem}
In Section \ref{subsec:A-Hessian-model}, we study some automorphisms
of $X_{\l}$. We obtain that the nine above $9{\bf A}_{2}$-configurations
are in the same orbit under the action by $3$-torsion of the Mordell-Weil
group of a fibration of $X_{\l}$. 
\end{rem}

\section{Projective, Hessian and Weierstrass models of $X_{\protect\l}$}

\subsection{\label{subsec:NaturalEllFib}The natural elliptic fibration of $X_{\protect\l}$}

Let $D_{2}$ be the big and nef divisor on $X_{\l}$ which is the
pull back of a line in $\PP^{2}$. For $\l$ generic, the divisors
$D_{2},A_{1},A_{1}'\dots,A_{9},A_{9}'$ form a $\QQ$-basis $\mathcal{B}_{\QQ}$
of $\NS X_{\l})_{/\QQ}$; they generate an index $3^{6}$ sub-lattice
of $\NS X_{\l})$ (see also Lemma \ref{LEMMA:The-NS-latt-Discri54}).

Let $\mu:Y_{\l}\to\PP^{2}$ be the blow-up of the plane at the $9$
cusps of the sextic curve $C_{\l}$ and let $E_{1},\dots,E_{9}$ be
the exceptional curves over $p_{1},\dots,p_{9}$. The strict transform
by $\mu$ of the curve $C_{\l}$ is the smooth genus $1$ curve 
\[
\strict{C}_{\l}=\mu^{*}C_{\l}-2{\textstyle \sum_{j=1}^{9}}E_{i},\text{ such that }\strict{C}_{\l}^{2}=0.
\]
The surface $X_{\l}$ is the double cover of $Y_{\l}$ branched over
$\bar{C_{\l}}$; we denote by 
\[
\eta':X_{\l}\to Y_{\l}
\]
the double cover morphism (so that $\eta'^{*}E_{j}=A_{j}+A_{j}'$)
and by $F_{o}$ the ramification locus, so that $2F_{o}=\eta^{'*}\bar{C}_{\l}$.
Since $2F_{o}=\eta'^{*}\bar{C_{\l}}\equiv6D_{2}-2\sum_{j=1}^{9}(A_{j}+A_{j}')$,
we get 
\[
F_{o}\equiv3D_{2}-\left({\textstyle \sum_{j=1}^{9}}A_{j}+A_{j}'\right).
\]
Let $L\hookrightarrow\PP^{2}$ be a line; the curve $C_{\l}$ belongs
to the linear system 
\[
\d=|6L-2{\textstyle \sum_{j=1}^{9}}p_{j}|
\]
of sextic curves with a double point at points in $\mathcal{P}_{9}$.
One computes that this linear system is $1$ dimensional. Moreover
there exists a unique cubic curve $\Ca(\l)$ (called the Cayleyan
curve, see \cite{AD}) that contains the $9$ points in $\mathcal{P}_{9}$,
which is 
\begin{equation}
\Ca(\l):\,\,x^{3}+y^{3}+z^{3}-\tfrac{(\l^{3}+2)}{\l}xyz=0,\label{eq:cayleysian}
\end{equation}
so that $2\Ca(\l)\in\d$. The linear system $\d$ lifts to a base
point free linear system $\d'$ on $Y_{\l}$ with $\bar{C}_{\l}\in\d'.$
The linear system $\d'$ defines a morphism $\varphi':Y_{\l}\to\PP^{1}$
and induces an elliptic fibration 
\[
\varphi:X_{\l}\to\PP^{1}
\]
for which $F_{o}$ is a fiber, and which we call the \textit{natural
fibration}.

Let $p,q$ be the images by $\varphi$ of the strict transforms in
$X_{\l}$ of $\Ca(\l)$ and $C_{\l}$. In fact the surface $X_{\l}$
is the fiber product of the fibration $\varphi'$ and the quadratic
transformation $\PP^{1}\to\PP^{1}$ branched at $p,q$. Indeed both
maps $X_{\l}\to Y_{\l}$ and $Y_{\l}\times_{\PP^{1}}\PP^{1}$ has
the same branch locus in the rational surface $Y_{\l}$. \\
 The curves $A_{1},A_{1}',\dots,A_{9},A_{9}'$ are sections of $\varphi$,
and one can check that the curves $\g_{1},\g_{1}',\dots,\g_{9},\g'_{9}$
are also sections (see Figure \ref{subsec:confOfConics}). 
\begin{thm}
\label{Thm:24InTheFibers} The fibration $\varphi$ contracts the
$24$ $\cu$-curves $\varTheta_{j}$, $j\in\{1,\dots,24\}$ which
are above the $12$ conics $\mathcal{C}_{12}$. The singular fibers
of $\varphi$ are $8$ fibers of type $\tilde{\mathbf{A}}_{2}$, each
singular fiber is the union of $3$ curves $\varTheta_{j}$. For $\l$
generic, the fibration $\varphi$ has fibers with non-constant moduli
and the Mordell-Weil group of the fibration $\varphi$ is $\ZZ\times(\ZZ/3\ZZ)^{2}$. 
\end{thm}

\begin{proof}
The following four sextic curves 
\[
\begin{array}{c}
C_{123456}+C_{123789}+C_{456789},\,C_{124578}+C_{134679}+C_{235689},\\
C_{124689}+C_{135678}+C_{234579},\,C_{125679}+C_{134589}+C_{234678},
\end{array}
\]
belong to the linear system $\d$ of sextic curves that have multiplicity
$2$ at the points in $\mathcal{P}_{9}$; actually their singularities
are nodes. By the results in the proof of Theorem \ref{thm:12Conics9Points},
the strict transform to $Y_{\l}$ of the above $4$ sextic form $4$
fibers of type $\tilde{\mathbf{A}}_{2}$, which lies in the étale
locus of $\eta$. Their strict transform on $X_{\l}$ is therefore
the union of eight fibers of type $\tilde{\mathbf{A}}_{2}$. A fiber
of type $\tilde{\mathbf{A}}_{2}$ contributes to $3$ in the Euler
characteristic of $X_{\l}$, which is equal to $24$. Since there
are $8\tilde{\mathbf{A}}_{2}$ singular fibers, the fibration has
no other singular fibers. The $24$ curves $\varTheta_{j}$ above
the $12$ conics are in the fibers, thus are contracted by $\varphi$.

 The strict transform $\He(\l)$ on $X_{\l}$ of $\Ca(\l)$ is smooth,
of genus $1$ (we will see that this is the Hessian of the curve $E_{\l}$,
thus the notation; see also Remark \ref{rem:doublecoverCayleysian}).
Since $\He(\l)\cdot F_{o}=0$, we have that $\He(\l)\equiv F_{o}$.
The curve $F_{o}$ is isomorphic to $E_{\l}$. For generic $\l$,
the curves $\Ca(\l)$ and $E_{\l}$ have distinct $j$-invariants,
thus the fibers of $\varphi$ have a non-constant moduli. Since the
fibration is not isotrivial, results of Shioda (see \cite[Corollary 1.5]{Shioda})
apply and tell that the Mordell-Weil group of sections of $\varphi:X_{\l}\to\PP^{1}$
has rank $1=19-(2+8(3-1))$.

In fact, elliptic fibrations of K3 surfaces are classified by Shimada
in \cite{Shimada}. A table with the $3278$ possible cases is available
in \cite{Shimada1}. Our fibration is case number $2373$ in that
table, where one can find moreover that the torsion part of its Mordell-Weil
group is isomorphic to $(\ZZ/3\ZZ)^{2}$. 
\end{proof}

\subsection{\label{subsec:Two-polarizations-and}Two polarizations and a degree
$8$ projective model }

The divisor 
\[
D_{14}=4D_{2}-\left({\textstyle \sum_{j=1}^{9}}A_{j}+A_{j}'\right)
\]
is linearly equivalent to $D_{2}+F$ ($F$ a fiber of $\varphi$)
and is effective. Let us define $D_{8}=D_{14}-(A_{1}+A_{1}')$. 
\begin{prop}
\label{prop:degree8Model}The divisors $D_{8}$ and $D_{14}$ are
ample of square $D_{8}^{2}=8,$ $D_{14}^{2}=14$. The linear system
$|D_{8}|$ is base point free, non-hyperelliptic, and defines an embedding
$X_{\l}\hookrightarrow\PP^{5}$ as a degree $8$ surface. For $d\in\NN^{*}$,
let $n_{d}$ be the number of $\cu$-curves of degree $d$ for $D_{8}$.
The series $\sum n_{d}T^{d}$ begins with 
\[
32T+20T^{2}+334T^{4}+576T^{5}+880T^{6}+8640T^{7}+17784T^{8}...,
\]
in particular $X_{\l}$ contains $32$ lines and $20$ conics. 
\end{prop}

\begin{proof}
Let $B$ be a $\cu$-curve such that $D_{14}B\leq0$. Since $D_{2}$
is effective and $D_{2}^{2}>0$, one has $D_{2}B\geq0$, moreover
since $F$ is a fiber, $FB\geq0$ and we must have $D_{2}B=0=FB$.
That implies that $B$ is an irreducible component of a singular fiber,
ie $B=\varTheta_{j}$ for some $j\in\{1,\dot{,24\}}$. But since $D_{2}\varTheta_{j}=2$
for all $j$, such a curve $B$ cannot exist, thus $D_{14}$ is ample.

Let us prove that $D_{8}$ is ample. We have 
\[
\g_{1}+\g_{1}'\equiv4D_{2}-2(A_{1}+A_{1}')-{\textstyle \sum_{j=1}^{9}}(A_{j}+A_{j}'),
\]
thus $D_{8}\equiv A_{1}+A_{1}'+\g_{1}+\g'_{1}$ and the divisor $D_{8}$
is effective. We check that $D_{8}A_{1}=D_{8}A_{1}'=$$D_{8}\g_{1}=D_{8}\g'_{1}=1$
and $D_{8}^{2}=8$, therefore $D_{8}$ is nef and big. Suppose that
there is a $\cu$-curve $B$ on $X_{\l}$ such that $D_{8}B=0$. Then
by the above expression of $D_{8}$, one has $A_{1}B=A_{1}'B=0$.
Let $L\hookrightarrow\PP^{2}$ be a line. For $j\in\{2,\dots,9\}$,
let us consider the linear system 
\[
\d_{j}=|4L-(p_{1}+p_{j}+{\textstyle \sum_{j=1}^{9}}p_{k})|
\]
of the quartic curves that go through the points in $\mathcal{P}_{9}$
and with multiplicity $2$ at $p_{1}$ and $p_{j}$. Using \verb"LinSys",
one can compute that for each $j>1$, the linear system $\d_{j}$
is a pencil of curves and the base points set is $\mathcal{P}_{9}$.
Moreover, the generic element $\vartheta_{j}$ of $\d_{j}$ is an
irreducible curve of geometric genus $1$ which cuts $C_{\l}$ in
$\mathcal{P}_{9}$ and two more points. Thus we obtain that for each
$j>1$, the strict transform of $\vartheta_{j}$ is an irreducible
curve $\G_{j}$ such that 
\[
D_{8}\equiv\G_{j}+A_{j}+A_{j}'\text{ and }\G_{j}^{2}=2.
\]
Since $D_{8}B=0$, we obtain $A_{j}B=A_{j}'B=0$ for all $j\in\{1,\dots,9\}$.
Since the orthogonal of the classes $A_{j},A_{j}'$, $j\in\{1,\dots,9\}$
(on which $B$ belongs) is generated by $D_{2}$, the class of $B$
must be a multiple of $D_{2}$ and have positive square, which is
absurd. Therefore $D_{8}$ is ample.

Suppose that there is a fiber $F'$ such that $D_{8}F'\in\{1,2\}$.
Observe that by using the expression for $D_{8}$, we get that $F'\Gamma_{j}=0,1,2$.
If $F'\Gamma_{j}=0$, then $\Gamma_{j}$ is contained in a fiber of
the fibration determined by $F'$, but this is not possible since
$\Gamma_{j}^{2}=2$. If $F'\Gamma_{j}=1$, then $\Gamma_{j}$ is a
section of the fibration so is a rational curve, but again this is
not possible. If $F'\Gamma_{j}=2$ (we can assume that this holds
for all $j$, otherwise we are in a previous case), then $F'$ is
in the orthogonal complement of the $A_{j},A_{j}'$ but this is not
possible since this is generated by $L$, which is of square $2$.
Therefore there are no such fiber $F'$ and using \cite{SaintDonat},
we obtain that the linear system $|D_{8}|$ is base-point free and
gives an embedding of $X_{\l}$.

With respect to the divisor $D_{8}$, the degrees of the curves $A_{1},A_{1}',\g_{1},\g_{1}'$
equal $2$ and the degrees of curves $A_{i},A_{i}',\,i\geq2$ is $1$.
For the assertions on the number of rational curves of degree $d\leq8$
we used an algorithm (see e.g.~\cite{Roulleau}), which computes
the classes of $\cu$-curves in $\NS X_{\l})$ of given degrees with
respect to a fixed ample class. 
\end{proof}
Proceeding in a similar way as in the proof of Proposition \ref{prop:degree8Model},
we obtain: 
\begin{prop}
Let $i,j\in\{1,\dots,9\}$, $i\neq j$. The divisor 
\[
D_{i,j}=D_{14}-(A_{i}+A_{i}'+A_{j}+A_{j}')
\]
is nef of square $2$ and the linear system $|D_{i,j}|$ is base point
free. 
\end{prop}

One can compute that the intersection with $D_{ij}$ is $0$ for the
$10$ curves $\t_{ijklmn},\t'_{ijklmn}$ (where $\{k,l,m,n\}\subset\{1,\dots,9\}$
is a set of $4$ elements such that the conic $C_{ijklmn}$ exists),
and for the $\cu$-curve which is the strict transform on $X_{\l}$
of the line through cusps $p_{i},p_{j}$.

\subsection{\label{subsec:A-Hessian-model}A Hessian model of the natural fibration
of $X_{\protect\l}$}

\subsubsection{The generic fiber of the elliptic fibration $\varphi$ and $18$
rational points}

Let $f_{\l}$ be the equation of the $9$ cuspidal sextic $C_{\l}$
which is the dual of $E_{\l}$, and let $c_{\l}$ be the equation
of the Cayleyan elliptic curve $\Ca(\l)$ (see equation \eqref{eq:cayleysian}),
the unique cubic that goes through the $9$ cusps of the sextic curve
$C_{\l}$.

We recall (see Section \ref{subsec:NaturalEllFib}) that $Y_{\l}$
is the blow-up of the plane at the $9$ points in $\mathcal{P}_{9}$;
it has a natural elliptic fibration $\varphi'$, coming from the pencil
of sextic curves which have double points at the $9$ points in $\mathcal{P}_{9}$,
pencil which is generated by $C_{\l}$ and $2\Ca(\l)$. A singular
model of $Y_{\l}$ is therefore obtained as the surface in $\PP^{1}\times\PP^{2}$
with equation $uf_{\l}-vc_{\l}^{2}=0$, where $u,v$ are the coordinates
of $\PP^{1}$. The projection onto $\PP^{1}$ induces the fibration
$\varphi':Y_{\l}\to\PP^{1}$. A singular model of the $K3$ surface
$X_{\l}$ is the surface $X_{\l}^{sing}$ in $\PP^{1}\times\PP^{2}$
with equation $u^{2}f_{\l}-v^{2}c_{\l}^{2}=0$; again the projection
onto $\PP^{1}$ induces the natural fibration $X_{\l}^{sing}\to\PP^{1}$,
where the generic fibers are $9$-nodal sextic curves.

In order to obtain a smooth model of $X_{\l}$, let us consider the
linear system $L_{4}(\mathcal{P}_{9})$ of quartics that contain the
$9$ cusps. The linear system $L_{4}(\mathcal{P}_{9})$ has (projective)
dimension $5$ and defines a rational map $\phi:\PP^{2}\dashrightarrow\PP^{5}$
not defined on $\mathcal{P}_{9}$. One computes that the image of
$X_{\l}^{sing}$ by the rational map 
\[
(i_{d},\phi):\PP^{1}\times\PP^{2}\dashrightarrow\PP^{1}\times\PP^{5}
\]
is a smooth model of $X_{\l}$; from Section \ref{subsec:NaturalEllFib},
the images of the cusps are the $18$ $\cu$-curves on $X_{\l}$ forming
a $9{\bf A}_{2}$-configuration. Taking the generic point over $\PP^{1}$,
one get a smooth genus $1$ curve in $\PP_{/\QQ(t)}^{5}$ (where $t=\frac{u}{v}$).
That curve $E_{K3}$ has naturally $18$ rational points, corresponding
to the $18$ $\cu$-curves. Using Magma, we computed a Hessian model
$E_{K3}\hookrightarrow\PP_{/\QQ(t)}^{2}$, which is 
\begin{thm}
\label{thm:AHessian-model-of}A model of the generic fiber of the
fibration $X_{\l}\to\PP^{1}$ is 
\begin{equation}
E_{K3}\,\,\,\,\,x^{3}+y^{3}+z^{3}+\frac{\l^{3}(t^{2}+3)-4t^{2}}{\l^{2}(t^{2}-1)}xyz=0.\label{eq:ofEK3}
\end{equation}
\end{thm}

The elliptic curve $E_{K3}$ contains the $9$ obvious $3$-torsion
points 
\[
\begin{array}{c}
Q_{1}=(0:-1:1),\,\,Q_{2}=(-1:0:1),\,\,Q_{3}=(-1:1:0),\hfill\\
Q_{4}=(0:-\o:1),\,\,Q_{5}=(\o+1:0:1),\,\,Q_{6}=(-\o:1:0),\hfill\\
Q_{7}=(0:\o+1:1),\,\,Q_{8}=(-\o:0:1),\,\,Q_{9}=(\o+1:1:0)
\end{array}
\]
(where $\o^{2}+\o+1=0$; we take $Q_{1}$ as the neutral element)
and the following $9$ points 
\[
\begin{array}{c}
P_{1}=(-2t:\l(t+1):\l t+\l),\hfill\\
P_{2}=(\l(t-1):-2t:\l t+\l),\hfill\\
P_{3}=(-\l t-\l:-\l t+\l:2t),\hfill\\
P_{4}=((2\o+2)t:\l(\o t+\o):\l t+\l)\hfill\\
P_{5}=(\l(\o+1)(-t+1):-2\o t:\l t+\l),\hfill\\
P_{6}=((\o+1)\l(t+1):\o\l(-t+1):2t),\hfill\\
P_{7}=(-2\o t:-\l(\o+1)(t+1):\l t-\l),\hfill\\
P_{8}=(\l\o(t-1):(2\o+2)t:\l t+\l),\hfill\\
P_{9}=(-\l\o(t-1):(\o+1)\l(t-1):2t).\hfill
\end{array}
\]
Together, these $18$ points are the above-mentioned rational points
of $E_{K3/\QQ(\o,t)}$ corresponding to the $18$ sections of the
fibration of $X_{\l}$. 
\begin{rem}
For the neutral element of $E_{K3}$, let us choose $Q_{1}$. One
can check that the point $P_{k}$ is the translate of $P_{1}$ by
the $9$ torsion point $Q_{k}$ ($k\in\{1,\dots,9\}$), i.e. $P_{k}=P_{1}+Q_{k}$. 
\end{rem}

\subsubsection{A smooth model of $X_{\protect\l}$ in $\protect\PP^{1}\times\protect\PP^{2}$ }

By taking the homogenization of the generic fiber $E_{K3}$ in \eqref{eq:ofEK3},
we get a natural model of the K3 surface $X_{\l}$ as 
\begin{equation}
\l^{2}(u^{2}-v^{2})(x^{3}+y^{3}+z^{3})+(\l^{3}(u^{2}+3v^{2})-4u^{2})xyz=0.\label{eq:HessianK3}
\end{equation}
in the space $\PP^{1}\times\PP^{2}$ (with coordinates $u,v;x,y,z$,
where $t=\frac{u}{v}$). That model is smooth, and the generic fibers
are smooth cubic curves, by contrast with the previous model $X_{\l}^{sing}$.
We denote by $(P)$ the section in $X_{\l}$ corresponding to the
points $P\in E_{K3}$. Using Magma, it is then possible to obtain
the equations of the $\cu$-curves (also sections) $A_{j}=(Q_{j})$,
resp. $A_{j}'=(P_{i})$, which are on $X_{\l}\hookrightarrow\PP^{1}\times\PP^{2}$.
We can check that: 
\begin{lem}
The $9$ curves $A_{j}+A_{j}'$ ($j\in\{1,\dots,9\}$) form a $9\mathbf{A}_{2}$-configuration. 
\end{lem}

\begin{proof}
We use the equations of the $\cu$-curves $A_{j},A_{j}'$ in the model
$X_{\l}\subset\PP^{1}\times\PP^{2}$ to check that $A_{j}A_{j}'=1$
and $A_{j}A_{k}'=A_{j}A_{k}=A_{j}'A_{k}'=0$ for $k\neq j$. In fact,
one already knows that $3$-torsion sections are disjoint by \cite[VII, Proposition 3.2]{Miranda}
(thus the sections $A_{j}'$, being translated of the group of $3$-torsion
sections, are also disjoint). 
\end{proof}
One can check moreover that the $9$ intersection points of $A_{j}$
with $A_{j}'$ for $i=1,\dots,9$ are on the fiber over $0$ of the
fibration $\varphi$, fiber which is isomorphic to $E_{\l}$. Using
the addition law on the elliptic curve $E_{K3}$, one can find other
points of $E_{K3}$, and therefore sections of $\varphi$. By example
the following points 
\[
\begin{array}{c}
R_{1}=(-2t:\l(t-1):\l t+\l),\hfill\\
R_{2}=(\l(t+1):-2t:\l t-\l),\hfill\\
R_{3}=(-\l t+\l:-\l t-\l:2t),\hfill\\
R_{4}=((2\o+2)t:\l\o(t-1):\l t+\l),\hfill\\
R_{5}=(-\l(\o+1)(t+1):-2\o t:\l t-\l),\hfill\\
R_{6}=((\o+1)\l(t-1):-\o\l(t+1):2t),\hfill\\
R_{7}=(-2\o t:\l(\o+1)(-t+1):\l t+\l),\hfill\\
R_{8}=(\l\o(t+1):(2\o+2)t:\l t-\l),\hfill\\
R_{9}=(-\o\l t+\o\l:(\o+1)(\l t+\l):2t),\hfill
\end{array}
\]
are the points $R_{i}=-P_{1}+Q_{i},$ $i\in\{1,\dots,9\}$. We have 
\begin{lem}
The $9$ curves $A_{i},\,(R_{i})$, $i=1,...,9$, form a $9{\bf A}_{2}$-configurations. 
\end{lem}

\begin{proof}
The curves $A_{i}$ (resp. $(R_{i})$) are images of the curves $A_{i}'$
(resp. $A_{i}$) by the translation by $-A_{1}'$ and we already know
that the curves $A_{i},A_{i}'$ form a $9{\bf A}_{2}$-configuration. 
\end{proof}
We already know that the fiber at $0$ of the elliptic K3 surface
$\varphi:X_{\l}\to\PP^{1}$ is isomorphic to $E_{\l}$, moreover: 
\begin{rem}
\label{rem:doublecoverCayleysian}From equation \ref{eq:HessianK3},
the fiber at $\infty$ of $X_{\l}\to\PP^{1}$ is the elliptic curve
\[
\He(\l):\,\,\,\,x^{3}+y^{3}+z^{3}+\tfrac{(\l^{3}-4)}{\l^{2}}xyz=0,
\]
which is in fact the Hessian of the curve $E_{\l}$. The $j$-invariants
of $\Ca(\l)$ and $\He(\l)$ are distinct, in particular these curves
are not isomorphic. The Cayleyan curve $\text{Ca}(\l)$ is the quotient
of Hessian $\He(\l)$ of $E_{\l}$ by a $2$-torsion point; in particular
the two curves are $2$-isogeneous, see \cite{AD}. 
\end{rem}

\subsubsection{\label{subsec:A-CompletInterModelInP5}A degree 8 non-complete intersection
model in $\protect\PP^{5}$}

One can check that the map from $X_{\l}\subset\PP^{1}\times\PP^{2}$
to $\PP^{5}$ obtained as the product of the identity map of $\PP^{1}$
with the Segre embedding composed with the projection to $\PP^{5}$,
is an embedding with image a degree $8$ K3 surface in $\PP^{5}$
defined by the following $5$ equations:

\[
\begin{array}{c}
-U_{2}U_{4}+U_{1}U_{5},\,\,\,\,-U_{2}U_{3}+U_{0}U_{5},\,\,\,\,-U_{1}U_{3}+U_{0}U_{4},\hfill\\
\lambda^{2}(U_{0}^{2}U_{3}-U_{3}^{3}+U_{1}^{2}U_{4}-U_{4}^{3})+(\lambda^{3}-4)U_{0}U_{1}U_{5}\\
\hfill+\lambda^{2}(U_{2}^{2}U_{5}+3\lambda U_{3}U_{4}U_{5}-U_{5}^{3}),\\
\lambda^{2}(U_{0}^{3}+U_{1}^{3}+U_{2}^{3}-U_{0}U_{3}^{2})+(\lambda^{3}-4)U_{0}U_{1}U_{2}\\
\hfill+\lambda^{2}(-U_{1}U_{4}^{2}+3\lambda U_{0}U_{4}U_{5}-U_{2}U_{5}^{2}),
\end{array}
\]
in particular this is not a complete intersection (in fact, by using
\cite[Chapter VIII, Exercice 11]{Beauville}, a K3 surface has a degree
$8$ smooth model which is not a complete intersection if and only
if it has a smooth model in $\PP^{1}\times\PP^{2}$ of bi-degree $(2,3)$).

Using the images of the known $\cu$-curves on $X_{\l}$, one finds
that this surface in $\PP^{5}$ contains at least $33$ lines.

The involution $\s'$ defined by $u\to(-u_{0}:-u_{1}:-u_{2}:u_{3}:u_{4}:u_{5})$
acts on $X_{\l}\hookrightarrow\PP^{5}$. Also one can check that the
order $3$ automorphisms 
\[
\begin{array}{c}
\a_{1}:u\to(u_{1}:u_{2}:u_{0}:u_{4}:u_{5}:u_{3}),\hfill\\
\a_{2}:u\to(u_{0}:\o u_{1}:\o^{2}u_{2}:u_{3}:\o u_{4}:\o^{2}u_{5})
\end{array}
\]
act on $X_{\l}$ and so does the involution 
\[
\b:u\to(u_{0}:u_{2}:u_{1}:u_{3}:u_{5}:u_{4}).
\]
The fixed point set of $\a_{1}$ is the union of $6$ points, the
fixed point set of $\beta$ is a genus $2$ curve. The involution
$\s'\b$ is symplectic. Using the above equations of $X_{\l}$ and
the equations of curves $A_{1},\dots,A_{9}'$ in $\PP^{5}$, one can
check that the automorphism group $G_{\mathcal{C}}$ that preserve
globally the $9{\bf A}_{2}$-configuration $A_{1},\dots,A_{9}'$ contains
the group $G_{18}$ isomorphic to $\ZZ_{3}\rtimes S_{3}$ generated
by $\a_{1},\a_{2},\s'\b$. We prove in Section \ref{subsec:stabilizerGroupFirstConf}
that $G_{18}$ has index $2$ in the group $G_{\mathcal{C}}$ preserving
the configuration $A_{1},\dots,A_{9}'$.

\subsection{\label{subsec:A-Weierstrass-equation}A Weierstrass equation}

We recall that $\o$ is such that $\o^{2}+\o+1=0$. Let us prove the
following result 
\begin{thm}
\label{thm:TheWeiModel}A minimal Weierstrass model of the elliptic
K3 surface $X_{\l}$ is the elliptic curve 
\[
E_{1/\QQ(\o,t)}:\,\,\,y^{2}=x^{3}-\frac{1}{48}Ax+\frac{1}{864}B,
\]
where the three polynomials $A,B,D$ in $\QQ(\o)(t)$ (of respective
degree $8$, $12$ and $8$) are defined in the Appendix. The $8$
singular fibers $\tilde{\mathbf{A}}_{2}$ of $X_{\l}$ are over the
$8$ zeros of $D$. 
\end{thm}

\begin{proof}
A direct computations gives that the $j$-invariant of the elliptic
curve $E_{K3}$ is

\[
j=-\frac{A^{3}}{(\l^{2}(\l^{3}-1)D)^{3}},
\]
where the formulas for the polynomials $A$ and $D$ in $t$ are given
in the Appendix. For any $j\notin\{0,1728\}$, the elliptic curve
\[
E_{0}(j):\,\,\,\,y^{2}=x^{3}-\frac{1}{48}\frac{j}{j-1728}x+\frac{1}{864}\frac{j}{j-1728}
\]
has $j$-invariant equal to $j$. In our case, we compute that we
have 
\[
\frac{j}{j-1728}=\frac{A^{3}}{B^{2}},
\]
where the polynomial $B$ is also defined in the Appendix. By taking
the change of variables 
\[
x'=u^{2}x,\,y'=u^{3}y
\]
with $u=(B/A)^{1/2}$ in the equation of $E_{0}(j)$, we obtain the
elliptic curve $E_{1}$. The curve $E_{1}$ has also its $j$-invariant
equals to $j$, is also in Weierstrass form, but its coefficients
are coprime degree $8$ and $12$ polynomials in $t$. The discriminant
of the equation of $E_{1}$ is 
\[
\Delta=-(\l^{2}(\l^{3}-1)D)^{3},
\]
where $D$ is a product of $8$ degree $1$ polynomials in $t$. According
to \cite[Table IV.3.1]{Miranda}, the associated elliptic surface
is a K3 surface with $8$ singular fibers of type $\tilde{{\bf A}}_{2}$.

Using Magma, we finally obtain an isomorphism defined over $\QQ(\o,t)$
between the Hesse model $E_{K3}$ and the Weierstrass model $E_{1}$. 
\end{proof}

\section{On automorphisms of the K3 surface $X_{\protect\l}$}

\subsection{\label{subsec:On-the-Mordell-Weil}On the Mordell-Weil lattice of
the elliptic fibration $\varphi$}

For a point $P\in E_{K3}(\QQ(\o,t))$, let us denote by $(P)\hookrightarrow X_{\l}$
the corresponding section of $\varphi:X_{\l}\to\PP^{1}$. We denote
by $\tau\in\aut(X_{\l})$ the automorphism which is the translation
by $A_{1}'$. We have 
\begin{thm}
Modulo torsion, the section $A_{1}'=(P_{1})$ generates the Mordell-Weil
lattice $\text{MWL}(X_{\l})$ of sections. 
\end{thm}

\begin{rem}
One can compute the classes in $\NS X{}_{\l})$ of the curves $(R_{i})$
(which are the translate by $(-P_{1})$ of the curves $A_{i}$); we
give these classes in the Appendix. Using that knowledge, we get the
matrix representation on $\NS X_{\l})$ of the action of the automorphism
$\tau$. The characteristic polynomial of $\tau$ is $(T-1)^{3}(T^{2}+T+1)^{8}.$
\end{rem}

\begin{proof}
Let $O=A_{1}$ be the zero section, and let $F$ be a fiber of $\varphi:X_{\l}\to\PP^{1}$.
Using the knowledge of the action of the automorphism $\tau$ (which
is the translation by $(P_{1})$) on $\NS X_{\lambda})$, we get that
\[
(6P_{1})-2(3P_{1})+O\equiv6F
\]
in $\NS X_{\l})$, thus (see e.g.~\cite[Chapter III, Theorem 9.5]{Silverman})
$\left\langle 3P_{1},3P_{1}\right\rangle =6$ and $\left\langle P_{1},P_{1}\right\rangle =\frac{2}{3}$,
where $\left\langle \cdot,\cdot\right\rangle $ is the bilinear pairing
on $\text{MWL}(X_{\l})$ associated to the canonical height.

Let $\text{Triv}(X_{\l})$ be the lattice generated the zero section
and the fibers components of the fibration. The determinant formula
\cite[Corollary 6.39]{ShiodaSchutt} is 
\[
|\det\,\NS X_{\l})|=|\det\,\text{Triv}(X_{\l})|\cdot\det\,\text{MWL}(X_{\l})/|\text{MWL}(X_{\l})|^{2}.
\]
By Lemma \ref{LEMMA:The-NS-latt-Discri54}, we know that $|\det\,\NS X_{\l})|=54$.
We have moreover $\det\,\text{Triv}(X_{\l})=-3^{8}$ and $|\text{MWL}(X_{\l})|^{2}=3^{4}$,
thus we obtain that 
\[
\det\,\text{MWL}(X_{\l})=\tfrac{2}{3}.
\]
By Theorem \ref{Thm:24InTheFibers}, the group $\text{MWL}(X_{\l})$
has rank $1$; since $\left\langle P_{1},P_{1}\right\rangle =\frac{2}{3}$,
we conclude that $P_{1}$ generates $\text{MWL}(X_{\l})$ modulo torsion. 
\end{proof}
Using the action of $\tau$ and its powers, we can obtain more classes
in $\NS X_{\l})$ of the sections on the K3 surface $X_{\l}\to\PP^{1}$. 
\begin{rem}
We searched the $9{\bf A}_{2}$-configurations among a set of $45$
sections, but we obtained only the expected ones, i.e. the $9{\bf A}_{2}$-configuration
that are translate of the configuration $A_{i},A_{i}'$, $i\in\{1,\dots,9\}$.
Since these configurations are images of one configuration by an automorphism
(the translation by $A_{1}'$ and its multiples), these $9{\bf A}_{2}$-configurations
give the same generalized Kummer structure. 
\end{rem}

\subsection{\label{subsec:More-auto}More elements of the automorphism group
and another double plane model}

The K3 surface $X_{\l}$ is constructed as the minimal desingularization
of the double cover of the plane branched over the sextic curve $C_{\l}$.
Let $\s\in\aut(X_{\lambda})$ be the corresponding involution. By
construction $\s(A_{j})=A_{j}'$, $\s(A_{j}')=A_{j}$ and $\sigma$
preserves the fiber of the fibration $X_{\l}\to\PP^{1}$. From these
facts, we know the action of $\s$ on the Néron-Severi lattice, since
we know also the action of $\tau$, and one can compute that 
\begin{lem}
We have $\s=\tau\s\tau.$ 
\end{lem}

Let us recall that we denoted by $(R_{i})$ the sections corresponding
to the points $R_{i}=-P_{1}+Q_{i}$ in $E_{K3}$. We also have a model
$X_{\l}\subset\PP^{1}\times\PP^{2}$. From the equations of the curves
involved, we obtain that:\\
 a) the natural map $\pi_{2}:X_{\l}\to\PP^{2}$ induced by the projection
$\PP^{1}\times\PP^{2}\to\PP^{2}$ is a $2$ to $1$ map ,\\
 b) it contracts the $\cu$-curves $A_{1},\dots,A_{9}$ to the $9$
torsion base points $\mathcal{T}_{9}$ of the Hesse pencil $x^{3}+y^{3}+z^{3}-3\mu xyz=0,$
$\mu\in\PP^{1}$. \\
 c) for $j\in\{1,\dots,9\}$, the curves $A_{j}'$ and $(R_{j})$
are mapped to a line $L_{j}'$ that contains exactly one point of
$\mathcal{T}_{9}$, that line is therefore tangent to the sextic curve
at its other intersection points. The $9$ lines are in general position.

We can therefore compute the action of the involution $\s'\in\aut(X_{\lambda})$
corresponding to the double cover $\pi_{2}$ on the Néron-Severi lattice
(the curves $A_{j}'$ and $(R_{j})$ are exchanged and one can check
that the fiber is preserved; the classes of $(R_{j})$ in base $F,A_{j},A_{j}'$
are in the appendix). The action of $\s'$ on $X_{\l}\hookrightarrow\PP^{5}$
is given in sub-section \ref{subsec:A-CompletInterModelInP5}. We
get: 
\begin{lem}
We have $\tau=\s\sigma'$. 
\end{lem}

We compute moreover that the pull-back by $\pi_{2}$ of the $9$ points
of $\mathcal{T}_{9}$ are the irreducible curves $A_{1},\dots,A_{9}$,
in fact we have 
\begin{thm}
The surface $X_{\l}$ is the minimal desingularization of the double
cover of the plane branched over the sextic curve $C_{6}$ which is
the union of the elliptic curves 
\[
E_{\l}:\,\,\,\,x^{3}+y^{3}+z^{3}-3\l xyz=0
\]
and the Hessian of $E_{\l}$:
\[
\He(\l):\,\,\,\,x^{3}+y^{3}+z^{3}+\tfrac{(\l^{3}-4)}{\l^{2}}xyz=0.
\]
\end{thm}

\begin{proof}
Let $F_{0}$ and $F_{\infty}$ be the two fibers of the fibration
$X_{\l}\to\PP^{1}$ at $0$ and infinity. One computes that $\pi_{2}(F_{0})=E_{\l}$,
$\pi_{2}(F_{\infty})=\He(\l)$, and moreover 
\[
\pi_{2}^{*}(E_{\l})=2F_{0}+{\textstyle \sum_{j=1}^{9}}A_{j},\,\,\pi_{2}^{*}(\He(\l))=2F_{\infty}+{\textstyle \sum_{j=1}^{9}}A_{j},
\]
thus $E_{\l}+\He(\l)$ is the branch locus of the map $\pi_{2}$.
By Bézout Theorem, the singularities of $E_{\l}+\He(\l)$ are nodal
since the curves $E_{\l}$ and $\He(\l)$ meet at $\mathcal{T}_{9}$. 
\end{proof}
\begin{rem}
For each $j=1,\dots,9$, the images by $\pi_{2}$ of the two sections
$(-2P_{1}+Q_{j}),\,(2P_{1}+Q_{j})$ is the same quartic curve, which
is nodal with $3$ nodes. The images of $(-P_{1}+Q_{j}),\,(P_{1}+Q_{j})$
are the $9$ lines, their coordinates in the dual plane with basis
$x,y,z$ are the same as the points in $\mathcal{P}_{9}$. These $9$
lines are the inflection lines of the curve $E_{\l}$. 
\end{rem}

We recall that the Hesse configuration is the point-line configuration
$(9_{4},12_{3})$ of the $9$ points in $\mathcal{T}_{9}$ and the
$12$ lines $\mathcal{L}_{12}$ such that each line contains $3$
points in $\mathcal{T}_{9}$ and each point is on $4$ lines. 
\begin{prop}
\label{prop:The12linesHesse}The images by $\pi_{2}$ in $\PP^{2}$
of the $24$ irreducible components of the singular fibers of the
fibration $\varphi:X_{\l}\to\PP^{1}$ are the $12$ lines of the Hesse
configuration. 
\end{prop}

\begin{proof}
We give in Theorem \ref{thm:TheWeiModel} the $8$ points $p\in\PP^{1}$
such that the fiber $F_{p}$ over $p$ is singular. We are then able
to compute these singular fibers in $X_{\l}\subset\PP^{1}\times\PP^{2}$
and their images in $\PP^{2}$. 
\end{proof}
\begin{rem}
Using the elliptic curve $E_{K3}$, we obtain that the sub-group $\text{Tor}_{3}$
of order $3$ elements in the Mordell-Weil lattice $\text{MWL}(X_{\l})$
is generated by two order $3$ elements $t_{1},\,t_{2}$ which acts
on $\NS X_{\l})$ via 
\[
t_{1}(A_{j})=A_{\s j},t_{1}(A_{j}')=A_{\s j}',\,\,\,t_{2}(A_{j})=A_{\mu j},t_{2}(A_{j}')=A_{\mu j}'
\]
where $\sigma=(1,2,3)(4,5,6)(7,8,9)$ and $\mu=(1,4,7)(2,5,8)(3,6,9)$.
The elements of $\text{Tor}_{3}$ commute with $\sigma$ (and of course
with $\tau$). The action of $\text{Tor}_{3}$ is transitive on the
nine $9{\bf A}_{2}$-configurations we found in section \ref{subsec:-9new-9A2-configurations}. 
\end{rem}

\subsection{\label{subsec:stabilizerGroupFirstConf}On the stabilizer group of
natural $9{\bf A}_{2}$-configuration}

Recall that $X_{\l}$ is the minimal desingularization of the double
cover of $\PP^{2}$ ramified over the $9$-cuspidal sextic $C_{\l}$.
The strict transform on $X_{\l}$ of $C_{\l}$ is a smooth elliptic
curve isomorphic to $E_{\l}$. The linear system defined by that elliptic
curve defines an elliptic fibration which we denote $\varphi:X_{\l}\to\PP^{1}$,
and which we called the \textit{natural }fibration. The curves $A_{j},A_{j}'$
(above the cusps) are sections of $\varphi$, the $24$ $\cu$-curves
$\t_{J},\t_{J}'$ (for some set $J\subset\{1,\dots,9\}$ of order
$6$) above the $12$ conics are contained in the fibers (see Proposition
\ref{Thm:24InTheFibers}), their classes are given in the Appendix,
under a simpler labelling $\varTheta_{j},\,j\in\{1,\dots,24\}$. We
have 
\begin{prop}
\label{prop:TheNS} An integral basis $\mathcal{B}$ of the Néron-Severi
lattice of $X_{\lambda}$ is 
\[
\begin{array}{c}
F,A_{1},A_{1}',A_{2},A_{2}',A_{3},A_{3}',A_{4},A_{4}',A_{5},A_{5}',A_{6},A_{7},A_{7}',\\
\varTheta_{5},\varTheta_{14},\varTheta_{22},\varTheta_{23},\varTheta_{20},
\end{array}
\]
where $F$ is a fiber of the fibration $\varphi$. The discriminant
group $A_{\NS X_{\l})}\simeq\ZZ/2\ZZ\times(\ZZ/3\ZZ)^{3}$ is generated
by 
\[
\begin{array}{c}
w_{0}=\frac{1}{2}(0,1,1,0,0,0,0,1,1,0,0,0,1,1,1,1,1,0,0)_{\mathcal{B}},\hfill\\
v_{1}'=\frac{1}{3}(A_{3}+2A_{3}'+A_{5}+2A_{5}'+A_{7}+2A_{7}'),\hfill\\
v_{2}'=\frac{1}{3}(A_{2}+2A_{2}'+2A_{4}+A_{4}'+2A_{5}+A_{5}'+A_{7}+2A_{7}'),\\
v_{3}'=\frac{1}{3}(A_{1}+2A_{1}'+A_{4}+2A_{5}'+A_{7}+2A_{7}').\hfill
\end{array}
\]
\end{prop}

\begin{proof}
By Lemma \ref{LEMMA:The-NS-latt-Discri54}, the discriminant group
of the Néron-Severi lattice has order $54$. The lattice generated
by the elements in $\mathcal{B}$ has rank $19$ and discriminant
equal to $54$, thus these elements generate $\NS X_{\l})$. By taking
the inverse of the intersection matrix of the vectors in $\mathcal{B}$,
we get the generators of the discriminant group. 
\end{proof}
The non-immediate part of the Gram matrix of the basis $\mathcal{B}$
is the intersection matrix of the $5$ curves $\varTheta_{j},\,j=5,14,22,23,20$
with the curves in $\mathcal{B}$, whose matrix is:\vspace{1mm}

\mbox{%
{\footnotesize{}$\left(\begin{array}{ccccccccccccccccccc}
0 & 1 & 0 & 0 & 1 & 0 & 0 & 1 & 0 & 0 & 1 & 0 & 1 & 0 & -2 & 1 & 0 & 0 & 0\\
0 & 0 & 1 & 0 & 0 & 1 & 0 & 0 & 1 & 0 & 0 & 1 & 0 & 1 & 1 & -2 & 0 & 0 & 0\\
0 & 0 & 0 & 0 & 1 & 1 & 0 & 0 & 0 & 0 & 1 & 1 & 0 & 0 & 0 & 0 & -2 & 0 & 0\\
0 & 0 & 0 & 0 & 0 & 0 & 0 & 0 & 1 & 0 & 1 & 0 & 1 & 0 & 0 & 0 & 0 & -2 & 0\\
0 & 0 & 0 & 0 & 1 & 1 & 0 & 1 & 0 & 0 & 0 & 0 & 0 & 1 & 0 & 0 & 0 & 0 & -2
\end{array}\right).$ }%
}

\vspace{1mm}Let $G_{\mathcal{C}}$ be the automorphism sub-group
of $\aut(X_{\l})$ preserving globally the configuration $\mathcal{C}=A_{1},A_{1}',\dots,A_{9},A_{9}'$.
The proof of the following Proposition will also serve as a preliminary
for the proof of Theorem \ref{MAINthm} in the next section: 
\begin{prop}
\label{pro:The-group-Aut-preserveConf}The group $G_{\mathcal{C}}$
is isomorphic to $\ZZ_{2}\times(\ZZ_{3}\rtimes S_{3})$, where $\ZZ_{3}\rtimes S_{3}$
acts symplectically on $X_{\l}$. The center of $G_{\mathcal{C}}$
is generated by the non-symplectic involution $\s$ associated to
the double cover map $X_{\l}\to\PP^{2}$ branched over the cuspidal
sextic curve $C_{\l}$. 
\end{prop}

\begin{proof}
Let $\phi$ be an element of $G_{\mathcal{C}}$. Since it preserves
globally the configuration $\mathcal{C}$, it must map its orthogonal
complement (generated by $D_{2}$, the pull-back of a line) to itself.
There are $2^{9}9!$ bijective maps 
\[
\mu:\{A_{1},A_{1}',\dots,A_{9},A_{9}'\}\to\{A_{1},A_{1}',\dots,A_{9},A_{9}'\}
\]
which preserves the incidence relations of $\mathcal{C}$. Since a
linear map is defined by the images of the vectors of a basis, each
map $\mu$ extends to a linear automorphism $\phi_{\mu}:\,\NS X_{\l})\otimes\QQ\to\NS X_{\l})\otimes\QQ$
sending $D_{2}$ to $D_{2}$ and the configuration $\mathcal{C}$
to itself. The action of $\phi$ on $\NS X_{\lambda})$ must be one
of these maps. Using the integral basis $\mathcal{B}$ of Proposition
\ref{prop:TheNS}, we obtain that among all the possibilities, only
$864$ matrices in basis $\mathcal{B}$ of such $\phi_{\mu}$ are
in $GL_{19}(\ZZ)$. These $864$ matrices are the elements of a group
$G_{864}$ isomorphic to the product of $\ZZ/2\ZZ$ with $AGL_{2}(\FF_{3})$,
the affine linear group of the space $\FF_{3}^{2}$. The center of
$G_{864}$ has order $2$ and is generated by the matrix of the non-symplectic
involution $\sigma$ defined in section \ref{subsec:More-auto}. In
the appendix, we give two generators $g_{1},g_{2}$ (of respective
order $8$ and $6$) of the group $G_{432}\subset G_{864}$ isomorphic
to $AGL_{2}(\FF_{3})$. Their action on the $2$-torsion part of the
discriminant group $A_{\NS X_{\lambda})}$ is trivial, and one computes
that their action on the $3$-torsion part of the discriminant group
$A_{\NS X_{\lambda})}$ is by the matrices 
\[
\bar{g}_{1}=\left(\begin{array}{ccc}
1 & 1 & 1\\
2 & 1 & 0\\
1 & 0 & 0
\end{array}\right),\,\bar{g}_{2}=\left(\begin{array}{ccc}
1 & 0 & 0\\
0 & 2 & 0\\
0 & 0 & 1
\end{array}\right)
\]
in basis $v_{1},v_{2},v_{3}$. The group $G_{\text{Disc}}$ generated
by $\bar{g}_{1},\,\bar{g}_{2}$ is isomorphic to the symmetric group
$S_{4}$. The kernel of the map $G_{432}\to G_{\text{Disc}}$ is a
group $G_{18}$ isomorphic to $\ZZ_{3}\rtimes S_{3}$. Let $\text{Tran}(X_{\lambda})$
be the orthogonal complement of $\NS X_{\lambda})$ in $H^{2}(X_{\lambda},\ZZ)$;
it is a rank $3$ signature $(2,1)$ lattice. There exists an isomorphism
\[
\g:A_{\text{Tran}(X_{\lambda})}\to A_{\NS X_{\lambda})}
\]
between the discriminant groups, such that the quadratic forms of
the groups satisfy $q_{\text{Tran}(X_{\lambda})}\circ\g=-q_{\NS X_{\lambda})}$.
Since an element $\phi\in G_{18}$ acts trivially on $A_{\NS X_{\lambda})}$,
we can extend it to an isometry $\G$ of $H^{2}(X_{\l},\ZZ)$ by gluing
it with the identity map on $\text{Tran}(X_{\lambda})$ (see e.g.~\cite[Theorem 12]{ShiodaRemarks}
and references therein for an example of such construction). Since
$\G$ is the identity on the space $\text{Tran}(X_{\lambda})\otimes\CC$
containing the period, this is a Hodge isometry. Since $\G$ naturally
preserves the polarization $D_{14}=4D_{2}-\sum_{j}(A_{j}+A_{j}')$
(studied in Proposition \ref{prop:degree8Model}), it is effective.
Therefore we can apply the Torelli Theorem for K3 surfaces (see \cite[Chap. VIII, Theorem 11.1]{BHPVdV})
and conclude that there exists a unique $g\in\aut(X_{\lambda})$ such
that $g^{*}=\G$ on $\NS X_{\lambda})$.

By \cite[Corollary 3.3.5]{Huybrecht}, since the Picard number of
$X_{\l}$ is odd, the only Hodge isometry on $\text{Tran}(X_{\lambda})$
is $\pm I_{d}$. Suppose that an element $g$ of $G_{432}$ not contained
in $G_{18}$ comes from an automorphism of $X_{\l}$. From the definition
of $g$, its action $\bar{g}$ on $A_{\NS X_{\lambda})}$ is non trivial,
thus its action on $\text{Tran}(X_{\lambda})$ is also non-trivial.
The automorphism $g$ is therefore non symplectic and it acts by $-Id$
on $\text{Tran}(X_{\lambda})$. Thus $\bar{g}\in G_{\text{Disc}}$
acts on the discriminant group $A_{\text{Tran}(X_{\lambda})}$ by
$-Id$. However $-Id$ is not contained in $G_{\text{Disc}}\simeq S_{4}$:
a contradiction and $g$ cannot come from an automorphism of $X_{\l}$.

Since $\sigma$ commutes with the elements of $G_{432}$, the automorphism
group of the configuration is the direct product of $G_{18}$ and
$\left\langle \sigma\right\rangle \simeq\ZZ_{2}$. 
\end{proof}
\begin{rem}
We give a geometric interpretation of the some elements of the group
$G_{\mathcal{C}}\simeq\ZZ_{2}\times(\ZZ_{3}\rtimes S_{3})$ in sub-section
\ref{subsec:A-CompletInterModelInP5}: the group $G_{\mathcal{C}}$
contains the $9$ translations by $3$-torsion automorphisms (from
the fibration $\varphi:X_{\lambda}\to\PP^{1}$), and the involution
$\sigma$ from the double cover $X_{\lambda}\to\PP^{2}$. One can
recover the automorphisms in the subgroup $\ZZ_{3}\rtimes S_{3}$
of $G_{\mathcal{C}}$ as follows:\\
Let $A$ be the abelian surface and let $J_{A}$ be the order $3$
symplectic automorphism acting on $A$ such that $X_{\l}=\Km_{3}(A)$
is the minimal resolution of $A/J_{A}$, the exceptional locus being
$\mathcal{C}$. The automorphism $J_{A}$ fixes $9$ points $0=s_{1},\dots,s_{9}$
which form a group isomorphic to $(\ZZ_{3})^{2}$. The group generated
by the translation by the $s_{k}$ and the involution $[-1]:z\in A\to-z\in A$
is isomorphic to $\ZZ_{3}\rtimes S_{3}$. These isomorphisms induce
automorphisms on $A/J_{A}$, hence automorphisms on K3 surface $X_{\l}$;
these automorphisms preserve $\mathcal{C}$. 
\end{rem}

\subsection{\label{subsec:sendingAConfToAnother}An automorphism sending a $9{\bf A}_{2}$-configuration
to another}

Let 
\[
\mathcal{C}'=B_{1},B_{1}',\dots,B_{9},B_{9}'
\]
be one of the nine $9{\bf A}_{2}$-configurations found in section
\ref{subsec:-9new-9A2-configurations} (so that $B_{j}B_{j}'=1$,
$B_{j}B_{k}=0$ for $j\neq k$), namely we choose the $9{\bf A}_{2}$-configuration
\[
\mathcal{C}'=\varTheta_{12},\varTheta_{9},\varTheta_{14},\varTheta_{5},\varTheta_{16},\varTheta_{7},\varTheta_{2},\varTheta_{3},\g_{1},\g_{1}',\varTheta_{4},\varTheta_{1},\varTheta_{8},\varTheta_{15},\varTheta_{6},\varTheta_{13},\varTheta_{10},\varTheta_{11},
\]
where the classes of the $\cu$-curves $\varTheta_{j}$ (which are
the curves $\t_{i,j,k,m,n,o}$ more conveniently labelled and numbered)
and $\g_{1},\g_{1}'$ are in the Appendix. The eight ${\bf A}_{2}$-configurations
$B_{j},B_{j}',\,j\in\{1,\dots,4,6,\dots,9\}$, are contained in the
fibers of $\varphi$ and the ${\bf A}_{2}$-configuration $B_{5},B_{5}'$
is the strict transform under the double cover map of a the quartic
curve $Q_{1}$. 
\begin{thm}
\label{MAINthm}There exists an automorphism $f$ of $X_{\l}$ sending
the curve $A_{j}$ (resp. $A_{j}'$) to the curve $B_{j}$ (resp.
$B_{j}'$) for $j\in\{1,\dots,9\}$. In particular, the configuration
$\mathcal{C}=A_{1},A_{1}',\dots,A_{9},A_{9}'$ is sent by $f$ to
the configuration $\mathcal{C}'$, and therefore the configuration
$\mathcal{C}'$ is on the $\aut(X_{\lambda})$-orbit of $\mathcal{C}$. 
\end{thm}

\begin{proof}
For finding the automorphism $f$ we proceeded as in Proposition \ref{pro:The-group-Aut-preserveConf}:\\
 Let $\phi\in\aut(X_{\l})$ be an automorphism sending $\mathcal{C}$
to $\mathcal{C}'$. The orthogonal complement of $\mathcal{C}$ is
generated by $D_{2}$, the pull-back of a line in $\PP^{2}$ and the
orthogonal complements of $\mathcal{C}'$ is generated by the divisor
\[
D_{2}'=7D_{2}-2{\textstyle \sum_{j=1}^{j=9}}(A_{j}+A_{j}')-2(A_{1}+A_{1}'),
\]
of square $2$. Since $\phi$ must preserve the Néron-Severi lattice,
we have $\phi(D_{2})=D_{2}'$. There are $2^{9}9!$ bijective maps
\[
\mu:\{A_{1},A_{1}',\dots,A_{9},A_{9}'\}\to\{B_{1},B_{1}',\dots,B_{9},B_{9}'\}
\]
which preserves the incidence relations of $\mathcal{C}$ and $\mathcal{C}'$
and which extend uniquely to isometries $\tilde{\phi}_{\mu}$ of $\NS X_{\l})\otimes\QQ$.
If two of these maps $\tilde{\phi}_{1},\tilde{\phi}_{2}$ preserve
the lattice $\NS X_{\l})$, then $\tilde{\phi}_{2}^{-1}\tilde{\phi}_{1}$
also preserves that lattice, moreover it also preserves the configuration
$\mathcal{C}$. It is therefore an element of the group $G_{864}$
defined in the proof of Proposition \ref{pro:The-group-Aut-preserveConf}.
Thus the set of maps $\tilde{\phi}$ (among the maps $\tilde{\phi}_{\mu}$)
that preserve $\NS X_{\l})$ is the orbit of $\tilde{\phi}_{1}$ under
the action of $G_{864}$ on the left. The element $\tilde{\phi}_{1}$
defined by sending $A_{j}$ to $B_{j}$ and $A_{j}'$ to $B_{j}'$
preserves the lattice $\NS X_{\l})$ (its matrix $f_{\mathcal{B}}$
in basis $\mathcal{B}$ is given in the Appendix). Its action on the
discriminant group is trivial. Then, as in the proof of Proposition
\ref{pro:The-group-Aut-preserveConf}, we can extend $\tilde{\phi}_{1}$
to an Hodge isometry of $H^{2}(X_{\lambda},\ZZ)$.

Let us prove that this isometry is effective. The polarization (see
Proposition \ref{prop:degree8Model}) $D_{14}=4D_{2}-\sum(A_{j}+A_{j}')$
is sent to $D_{14}'=4D_{2}'-\sum(B_{j}+B_{j}')=D_{2}'+F_{1}$ where
$F_{1}$ is (the class of) a fiber of a fibration $f_{1}$ by Remark
\ref{rem:ExtraFib}. Using the classes of the curves in $\NS X_{\lambda})$,
it is easy to check that the curves $B_{j},B_{j}'$ are $18$ sections
of $f_{1}$.

The degree $7$ curve $Q\hookrightarrow$$\PP^{2}$ defined by{\footnotesize{}
\[
\begin{array}{c}
\l x^{7}-2\l^{2}x^{6}y+3\l^{3}x^{5}y^{2}-(2\l^{4}+\l)x^{4}y^{3}+\l^{5}x^{3}y^{4}+3\l^{3}x^{2}y^{5}\\
+\l^{2}y^{7}-2\l^{2}x^{6}z-1x^{5}yz+4\l x^{4}y^{2}z-(2\l^{5}+9\l^{2})x^{3}y^{3}z+(\l^{3}+1)x^{2}\\
y^{4}z-(3\l^{4}+2\l)xy^{5}z-2\l^{2}y^{6}z+3\l^{3}x^{5}z^{2}+4\l x^{4}yz^{2}+(3\l^{5}+4\l^{2})x^{3}y^{2}z^{2}\\
+(4\l^{3}-1)x^{2}y^{3}z^{2}+4\l xy^{4}z^{2}+\l^{2}y^{5}z^{2}-(2\l^{4}+\l)x^{4}z^{3}-(2\l^{5}+9\l^{2})x^{3}yz^{3}\\
+(4\l^{3}-1)x^{2}y^{2}z^{3}-(3\l^{4}+4\l)xy^{3}z^{3}+\l^{5}y^{4}z^{3}+\l^{5}x^{3}z^{4}+(\l^{3}+1)x^{2}yz^{4}\\
+4\l xy^{2}z^{4}+\l^{5}y^{3}z^{4}+3\l^{3}x^{2}z^{5}-(3\l^{4}+2\l)xyz^{5}+\l^{2}y^{2}z^{5}-2\l^{2}yz^{6}+\l^{2}z^{7}=0
\end{array}
\]
}is irreducible with singularities of multiplicity $4$ at the point
$p_{1}$ in $\mathcal{P}_{9}$, and with multiplicity $2$ at the
eight remaining points (the curve $Q$ was found by using \verb"LinSys").
The geometric genus of $Q$ is $1$ and $Q$ meets the branch locus
at two other points, so that its strict transform on $X_{\l}$ is
smooth of genus $2$ in the linear system $|D_{2}'|$. Thus $D_{2}'$
is effective, nef and $|D_{2}'|$ is base point free.

Let $C$ be an irreducible $\cu$-curve on $X_{\l}$. If $CF'=CD_{2}'=0$,
then $C$ is contained in a fiber of $f_{1}$ and is contracted by
$|D_{2}'|$. The second point implies that $C$ is one of the curves
$B_{j},B_{j}'$ (otherwise the Picard number of $X_{\l}$ would be
$>19$), but then one has $CF'=1$, which is a contradiction. Thus
$D_{14}'$ is ample and the isometry is effective.

We then conclude as in the proof of Proposition \ref{pro:The-group-Aut-preserveConf}
that there exists an automorphism of $X_{\l}$ such that its action
on the curves $A_{j},A_{j}'$ is as described. 
\end{proof}
We recall that for $k=1,...,9$, the point $p_{k}\in\mathcal{P}_{9}$
is the point over which $A_{k}+A_{k}'$ is contracted by the double
cover map $\eta:X_{\l}\to\PP^{2}$. 
\begin{rem}
\label{rem:ExtraFib}For any point $p_{k}$, the pencil of lines through
$p_{k}$ induces an elliptic fibration $f_{k}:X_{\l}\to\PP^{1}$ such
that the curves $A_{k},A_{k}'$ are sections and the curves $A_{j},A_{j}'$
for $j\neq k$ are contained in the fibers. The singular fibers of
$f_{k}$ are also $8\tilde{{\bf A}_{2}}$; from the known intersection
numbers of these fibers with the elements of basis $D_{2},A_{1},\dots,A_{9}'$,
we obtain that the class of a fiber is 
\[
F_{k}=D_{2}-(A_{k}+A_{k}')=\tfrac{1}{3}(F+{\textstyle \sum_{j=1}^{9}}(A_{j}+A_{j}'))-(A_{k}+A_{k}')=3D_{2}'-{\textstyle \sum_{j=1}^{9}}(B_{j}+B_{j}')
\]
where $F$ is a fiber of $\varphi$ and the $B_{j},B_{j}'$ are the
curve in the $9{\bf A}_{9}$-configuration found in Section \ref{subsec:-9new-9A2-configurations},
$D_{2}'$ being the generator of the orthogonal complement of the
$B_{j},B_{j}'$'s. The curves $B_{j},B_{j}'$ are sections of $f_{k}$. 
\end{rem}

So the geometric situation for the configuration $A_{1},A_{1}',\dots,A_{9},A_{9}'$
and fibration $f_{k}$ is very similar to the situation for $\mathcal{C}'$
and fibration $\varphi$ for which $8$ of the ${\bf A}_{2}$-configurations
in $B_{1},B_{1}',\dots,B_{9},B_{9}'$ are in the $8$ singular fibers
and the remaining one are sections.

\subsubsection{Aligned singularities of the union of the $12$ conics}

The fibration $\varphi$ has the remarkable property that the $18$
sections $A_{j},A_{j}'$ meet on the same fiber, which is isomorphic
to $E_{\l}$. It can be instructive to understand how the similar
result holds for the fibration $f_{1}$ (see Remark \ref{rem:ExtraFib})
and the curves $B_{j},B_{j}'$, this is the aim of this subsection,
which also gives an explanation why each line of the dual Hesse configuration
contains $4$ double points of the curve $\sum_{C\in\mathcal{C}_{12}}C$
and these double points form the set $\mathcal{P}_{12}$.

We recall that the $12$ conics in $\mathcal{C}_{12}$ meet in either
$4$ or $3$ points in $\mathcal{P}_{9}$. If two conics meet in $3$
points in $\mathcal{P}_{9}$ then the fourth intersection point is
an ordinary singularity of the union of the $12$ conics. Above the
$8$ conics that contain $p_{1}$ are the $16$ $\cu$-curves that
gives a $8{\bf A}_{8}$-configuration, which one can complete to a
$9{\bf A}_{2}$-configuration according to section \ref{subsec:-9new-9A2-configurations}.

The $8$ conics containing $p_{1}$ have the property that they meet
by pairs into $4$ points $q_{1},\dots,q_{4}$ not in $\mathcal{P}_{9}$
and that these $4$ points are on a line containing $p_{1}$. That
line $L_{1}$ is the tangent line to the cusp $p_{1}$ (it is one
of the lines of the dual Hesse configuration defined in Section \ref{subsec:The-Hesse,-dual}).
It meets the cuspidal sextic in $p_{1}$ (with multiplicity $3$)
and at points $r_{1},r_{2},r_{3}$. One can check that the fiber $F_{0}$
of $f_{1}$ which is the strict transform on $X_{\l}$ of $L_{1}$
is isomorphic to $E_{\l}$ (since we know the branch locus $F_{0}\to L_{1}$).

The $8$ points in $X_{\l}$ above $q_{1},\dots,q_{4}$ are the meeting
points of the curves in the $8{\bf A}_{2}$ configuration obtained
by taking the strict transform of the $8$ conics trough $p_{1}$.
Following the Figure \ref{figniceDiag}, the $9^{th}$ ${\bf A}_{2}$-configuration
also has its intersection point on that fiber $F_{0}$. In fact, the
fiber $F_{0}$ is the image by an automorphism of $X_{\l}$ of the
fiber over $0$ of $\varphi$.

The union of the dual Hesse configuration and the $12$ conics in
$\mathcal{C}_{12}$ is a line-conic arrangement with $9$ points of
multiplicity $9$, $12$ points of multiplicity $5$ and $72$ double
points, and with other remarkable properties studied in \cite{PS}.

\newpage

\section{\label{sec:Appendix} Appendix}

Let $D_{2}$ be the pull-back of a line by the double cover map $\eta:X_{\l}\to\PP^{1}$.
Let us define the following classes in the $\QQ$-basis $\mathcal{B}_{0}=(D_{2},A_{1},A_{1}',\dots,A_{9},A_{9}')$:
\[
B_{1}=2D_{2}-\tfrac{1}{3}\left({\textstyle \sum_{j=1}^{9}}2A_{j}+A_{j}'\right),\,\,B_{2}=2D_{2}-\tfrac{1}{3}\left({\textstyle \sum_{j=1}^{9}}A_{j}+2A_{j}'\right).
\]
We remark that $B_{1}^{2}=B_{2}^{2}=2$, $B_{1}B_{2}=5$ and $B_{1}+B_{2}=D_{14}$.
We have $D_{14}A_{j}=D_{14}A_{j}'=1$, therefore $B_{i}A_{j}\in\{0,1\}$,
$B_{i}A_{j}'\in\{0,1\}$. Using algorithms described in \cite{Roulleau},
we find that for $j\in\{1,\dots,9\}$, the classes of the curves $\gamma_{j},\gamma_{j}'$
above the quartic $Q_{j}$ are 
\[
\gamma_{j}=B_{1}-(A_{j}+A_{j}'),\,\gamma_{j}'=B_{2}-(A_{j}+A_{j}').
\]
It is easy to check that $\gamma_{j}^{2}=\gamma_{j}'^{2}=-2$, $\gamma_{j}\gamma_{j}'=1$,
and for $1\leq i\neq j\leq9$, we have $\gamma_{i}\gamma_{j}=\gamma_{i}'\gamma_{j}'=0$
and $\gamma_{i}\gamma_{j}'=3$. In fact, using that the image in $\PP^{2}$
of $\gamma_{j},\gamma_{j}'$ is a quartic curve that goes through
the points in $\mathcal{P}_{9}$ with a multiplicity $3$ at $p_{j}$,
one gets 
\[
4D_{2}\equiv\gamma_{j}+\gamma_{j}'+2(A_{j}+A_{j}')+{\textstyle \sum_{j=1}^{9}}(A_{j}+A_{j}').
\]

The classes in the $\QQ$-basis $\mathcal{B}_{0}=(L,A_{1},A_{1}',\dots,A_{9},A_{9}')$
of the $24$ $\cu$-curves $\t_{i,\dots,n},\t'_{i,\dots,n}$ above
the $12$ conics $C_{i,\dots,n}$ in $\mathcal{C}_{12}$ are {\footnotesize{}{}
\[
\begin{array}{c}
\t_{123456}=\frac{1}{3}(3,-2,-1,-2,-1,-2,-1,-1,-2,-1,-2,-1,-2,0,0,0,0,0,0),\\
\t'_{123456}=\frac{1}{3}(3,-1,-2,-1,-2,-1,-2,-2,-1,-2,-1,-2,-1,0,0,0,0,0,0),\\
\t_{123789}=\frac{1}{3}(3,-2,-1,-2,-1,-2,-1,0,0,0,0,0,0,-1,-2,-1,-2,-1,-2),\\
\t'_{123789}=\frac{1}{3}(3,-1,-2,-1,-2,-1,-2,0,0,0,0,0,0,-2,-1,-2,-1,-2,-1),\\
\t_{124578}=\frac{1}{3}(3,-2,-1,-1,-2,0,0,-2,-1,-1,-2,0,0,-2,-1,-1,-2,0,0),\\
\t'_{124578}=\frac{1}{3}(3,-1,-2,-2,-1,0,0,-1,-2,-2,-1,0,0,-1,-2,-2,-1,0,0),\\
\t_{124689}=\frac{1}{3}(3,-2,-1,-1,-2,0,0,-1,-2,0,0,-2,-1,0,0,-2,-1,-1,-2),\\
\t'_{124689}=\frac{1}{3}(3,-1,-2,-2,-1,0,0,-2,-1,0,0,-1,-2,0,0,-1,-2,-2,-1),\\
\t_{125679}=\frac{1}{3}(3,-2,-1,-1,-2,0,0,0,0,-2,-1,-1,-2,-1,-2,0,0,-2,-1),\\
\t'_{125679}=\frac{1}{3}(3,-1,-2,-2,-1,0,0,0,0,-1,-2,-2,-1,-2,-1,0,0,-1,-2),\\
\t_{134589}=\frac{1}{3}(3,-2,-1,0,0,-1,-2,-1,-2,-2,-1,0,0,0,0,-1,-2,-2,-1),\\
\t'_{134589}=\frac{1}{3}(3,-1,-2,0,0,-2,-1,-2,-1,-1,-2,0,0,0,0,-2,-1,-1,-2),\\
\t_{134679}=\frac{1}{3}(3,-2,-1,0,0,-1,-2,-2,-1,0,0,-1,-2,-2,-1,0,0,-1,-2),\\
\t'_{134679}=\frac{1}{3}(3,-1,-2,0,0,-2,-1,-1,-2,0,0,-2,-1,-1,-2,0,0,-2,-1),\\
\t_{135678}=\frac{1}{3}(3,-2,-1,0,0,-1,-2,0,0,-1,-2,-2,-1,-1,-2,-2,-1,0,0),\\
\t'_{135678}=\frac{1}{3}(3,-1,-2,0,0,-2,-1,0,0,-2,-1,-1,-2,-2,-1,-1,-2,0,0),\\
\t{}_{234579}=\frac{1}{3}(3,0,0,-2,-1,-1,-2,-2,-1,-1,-2,0,0,-1,-2,0,0,-2,-1),\\
\t'_{234579}=\frac{1}{3}(3,0,0,-1,-2,-2,-1,-1,-2,-2,-1,0,0,-2,-1,0,0,-1,-2),\\
\t{}_{234678}=\frac{1}{3}(3,0,0,-2,-1,-1,-2,-1,-2,0,0,-2,-1,-2,-1,-1,-2,0,0),\\
\t'_{234678}=\frac{1}{3}(3,0,0,-1,-2,-2,-1,-2,-1,0,0,-1,-2,-1,-2,-2,-1,0,0),\\
\t{}_{235689}=\frac{1}{3}(3,0,0,-2,-1,-1,-2,0,0,-2,-1,-1,-2,0,0,-2,-1,-1,-2),\\
\t'_{235689}=\frac{1}{3}(3,0,0,-1,-2,-2,-1,0,0,-1,-2,-2,-1,0,0,-1,-2,-2,-1),\\
\t_{456789}=\frac{1}{3}(3,0,0,0,0,0,0,-1,-2,-1,-2,-1,-2,-2,-1,-2,-1,-2,-1),\\
\t'_{456789}=\frac{1}{3}(3,0,0,0,0,0,0,-2,-1,-2,-1,-2,-1,-1,-2,-1,-2,-1,-2).
\end{array}
\]
}{\footnotesize\par}

We also denote by $\varTheta_{j},\,j=1,\dots,24$ these curves in
the order of the above list.

For $k=1,...,9$, the quartic curves $Q_{k}$ through $\mathcal{P}_{9}$
that have a multiplicity $3$ singular point at $p_{k}$ are:

{\footnotesize{}{} 
\[
\begin{array}{c}
Q_{1}:\,\,x^{4}-2\l x^{3}y+3\l^{2}x^{2}y^{2}-(\l^{3}+1)xy^{3}+\l y^{4}-2\l x^{3}z+(-\l^{3}+1)xy^{2}z-2\l y^{3}z\hfill\\
+3\l^{2}x^{2}z^{2}+(-\l^{3}+1)xyz^{2}+(\l^{4}+2\l)y^{2}z^{2}-(\l^{3}+1)xz^{3}-2\l yz^{3}+\l z^{4}=0,\\
Q_{2}:\,\,x^{4}-(\l^{3}+1)/\l x^{3}y+3\l x^{2}y^{2}-2xy^{3}+1/\l y^{4}-2x^{3}z+(-\l^{3}+1)/\l x^{2}yz-2y^{3}z\\
+(\l^{3}+2)x^{2}z^{2}+(1-\l^{3})/\l xyz^{2}+3\l y^{2}z^{2}-2xz^{3}-(\l^{3}+1)/\l yz^{3}+z^{4}=0,\\
Q_{3}:\,\,x^{4}-2x^{3}y+(\l^{3}+2)x^{2}y^{2}-2xy^{3}+y^{4}-(\l^{3}+1)/\l x^{3}z+(1-\l^{3})/\l x^{2}yz\hfill\\
+(1-\l^{3})/\l xy^{2}z-(\l^{3}+1)/\l y^{3}z+3\l x^{2}z^{2}+3\l y^{2}z^{2}-2xz^{3}-2yz^{3}+1/\l z^{4}=0,\\
Q_{4}:\,\,x^{4}+(2\omega+2)\l x^{3}y+3\omega\l^{2}x^{2}y^{2}-(\l^{3}+1)xy^{3}-(\omega+1)\l y^{4}-2\omega\l x^{3}z\hfill\\
-(\omega^{2}\l^{3}+\omega+1)xy^{2}z-2\omega\l y^{3}z-(3\omega+3)\l^{2}x^{2}z^{2}+(\omega-\omega\l^{3})xyz^{2}\\
+(\l^{4}+2\l)y^{2}z^{2}-(\l^{3}+1)xz^{3}+(2\omega+2)\l yz^{3}+\omega\l z^{4}=0,\\
Q_{5}:\,\,x^{4}-(\omega^{2}\l^{3}+\omega^{2})/\l x^{3}y+3\omega\l x^{2}y^{2}-2xy^{3}-(\omega+1)/\l y^{4}-2\omega x^{3}z\hfill\\
+(1-\l^{3})/\l x^{2}yz-2\omega y^{3}z-((\omega+1)\l^{3}+2\omega+2)x^{2}z^{2}+(\omega-\omega\l^{3})/\l xyz^{2}\\
+3\l y^{2}z^{2}-2xz^{3}-\omega^{2}(\l^{3}+1)/\l yz^{3}+\omega z^{4}=0,\\
Q_{6}:\,\,x^{4}+(2\omega+2)x^{3}y+(\omega\l^{3}+2\omega)x^{2}y^{2}-2xy^{3}+\omega^{2}y^{4}-(\omega\l^{3}+\omega)/\l x^{3}z\hfill\\
+(1-\l^{3})/\l x^{2}yz-(\omega^{2}\l^{3}-\omega^{2})/\l xy^{2}z-(\omega\l^{3}+\omega)/\l y^{3}z\\
-(3\omega+3)\l x^{2}z^{2}+3\l y^{2}z^{2}-2xz^{3}+(2\omega+2)yz^{3}+\omega/\l z^{4}=0,\\
Q_{7}:\,\,x^{4}-2\omega\l x^{3}y-(3\omega+3)\l^{2}x^{2}y^{2}-(\l^{3}+1)xy^{3}+\omega\l y^{4}+(2\omega+2)\l x^{3}z\hfill\\
+(\omega-\omega\l^{3})xy^{2}z+(2\omega+2)\l y^{3}z+3\omega\l^{2}x^{2}z^{2}-(\omega^{2}\l^{3}-\omega^{2})xyz^{2}\\
+(\l^{4}+2\l)y^{2}z^{2}-(\l^{3}+1)xz^{3}-2\omega\l yz^{3}+\omega^{2}\l z^{4}=0,\\
Q_{8}:\,\,x^{4}-(\omega\l^{3}+\omega)/\l x^{3}y-(3\omega+3)\l x^{2}y^{2}-2xy^{3}+\omega/\l y^{4}+(2\omega+2)x^{3}z\hfill\\
+(1-\l^{3})/\l x^{2}yz+(2\omega+2)y^{3}z+(\omega\l^{3}+2\omega)x^{2}z^{2}-(\omega^{2}\l^{3}-\omega^{2})/\l xyz^{2}\\
+3\l y^{2}z^{2}-2xz^{3}-(\omega\l^{3}+\omega)/\l yz^{3}+\omega^{2}z^{4}=0,\\
Q_{9}:\,\,x^{4}-2\omega x^{3}y+(\omega^{2}\l^{3}-2\omega-2)x^{2}y^{2}-2xy^{3}+\omega y^{4}-(\omega^{2}\l^{3}+\omega^{2})/\l x^{3}z\hfill\\
+(1-\l^{3})/\l x^{2}yz+(-\omega\l^{3}+\omega)/\l xy^{2}z-(\omega^{2}\l^{3}+\omega^{2})/\l y^{3}z\\
+3\omega\l x^{2}z^{2}+3\l y^{2}z^{2}-2xz^{3}-2\omega yz^{3}+\omega^{2}/\l z^{4}=0,
\end{array}
\]
}where $\o^{2}+\o+1=0$. 

The matrices in basis $\mathcal{B}$ of the generators of the group
$G_{432}\simeq AGL_{2}(\FF_{3})$ preserving the natural $9{\bf A}_{2}$-configuration
$A_{1},A_{1}',\dots,A_{9},A_{9}'$ are 
\[
\mbox{\footnotesize\ensuremath{g_{1}=\left(\begin{array}{ccccccccccccccccccc}
1 & 0 & 0 & 0 & 0 & 0 & 0 & 1 & -2 & 0 & 0 & 0 & 0 & 0 & 0 & 1 & 0 & 1 & 0\\
0 & 0 & 0 & 0 & 0 & -1 & 1 & 0 & 0 & 0 & 0 & 0 & 0 & 0 & 1 & 0 & 0 & -1 & 0\\
0 & 0 & 0 & 0 & 0 & 0 & 1 & 0 & 0 & 0 & 0 & 0 & 0 & 0 & 1 & 0 & 0 & 0 & 0\\
0 & 1 & 0 & 0 & 0 & 1 & -1 & 1 & -1 & 0 & 0 & 0 & 0 & 0 & -1 & 0 & 0 & 1 & 0\\
0 & 0 & 1 & 0 & 0 & 2 & -2 & 1 & 0 & 0 & 0 & 0 & 0 & 0 & -1 & -1 & 0 & 1 & 0\\
0 & 0 & 0 & 0 & 0 & 4 & -2 & 1 & 0 & 1 & 0 & 0 & 0 & 0 & -1 & -1 & -1 & 2 & 0\\
0 & 0 & 0 & 0 & 0 & 2 & -1 & 1 & -1 & 0 & 1 & 0 & 0 & 0 & -1 & 0 & -1 & 1 & 0\\
0 & 0 & 0 & 0 & 0 & -1 & 1 & 0 & 1 & 0 & 0 & 0 & 1 & 0 & 0 & 0 & 1 & -1 & 0\\
0 & 0 & 0 & 0 & 0 & -2 & 2 & -1 & 1 & 0 & 0 & 0 & 0 & 1 & 1 & 0 & 1 & -1 & 0\\
0 & 0 & 0 & 1 & 0 & -1 & 0 & 0 & -1 & 0 & 0 & 0 & 0 & 0 & 0 & 1 & 0 & 0 & 0\\
0 & 0 & 0 & 0 & 1 & -2 & 0 & -1 & 0 & 0 & 0 & 0 & 0 & 0 & 0 & 1 & 0 & -1 & 0\\
0 & 0 & 0 & 0 & 0 & 2 & -1 & 0 & 0 & 0 & 0 & 1 & 0 & 0 & 0 & -1 & -1 & 1 & 0\\
0 & 0 & 0 & 0 & 0 & -3 & 2 & -1 & 1 & 0 & 0 & 0 & 0 & 0 & 1 & 1 & 1 & -2 & 0\\
0 & 0 & 0 & 0 & 0 & 0 & 1 & 0 & 1 & 0 & 0 & 0 & 0 & 0 & 0 & 0 & 0 & 0 & 0\\
0 & 0 & 0 & 0 & 0 & -2 & 1 & -1 & 2 & 0 & 0 & 0 & 0 & 0 & 1 & 0 & 1 & -2 & 0\\
0 & 0 & 0 & 0 & 0 & 1 & 1 & -1 & 2 & 0 & 0 & 0 & 0 & 0 & 1 & -1 & 0 & 0 & 0\\
0 & 0 & 0 & 0 & 0 & 3 & -3 & 0 & 0 & 0 & 0 & 0 & 0 & 0 & -1 & -1 & -1 & 1 & 0\\
0 & 0 & 0 & 0 & 0 & -4 & 2 & -2 & 1 & 0 & 0 & 0 & 0 & 0 & 1 & 1 & 1 & -2 & 1\\
0 & 0 & 0 & 0 & 0 & 2 & -1 & 1 & 1 & 0 & 0 & 0 & 0 & 0 & -1 & -1 & 0 & 1 & 0
\end{array}\right),} }
\]
\[
\mbox{\footnotesize\ensuremath{g_{2}=\left(\begin{array}{ccccccccccccccccccc}
1 & 0 & 0 & 0 & 0 & 0 & 0 & 1 & -2 & 0 & 0 & 0 & 0 & 0 & 0 & 1 & 0 & 1 & 0\\
0 & 0 & 0 & 0 & 0 & -1 & 1 & 0 & 0 & 1 & 0 & 0 & 0 & 2 & 0 & 0 & 1 & 0 & 0\\
0 & 0 & 0 & 0 & 0 & 0 & 1 & 0 & 0 & 0 & 1 & 0 & 0 & 1 & 0 & 0 & 0 & 0 & 0\\
0 & 0 & 0 & 0 & 0 & 1 & -1 & 1 & -1 & 0 & 0 & 0 & 0 & -1 & 0 & 0 & -1 & 0 & 0\\
0 & 0 & 0 & 0 & 0 & 2 & -2 & 1 & 0 & 0 & 0 & 0 & 0 & -2 & 0 & 0 & -1 & 0 & 0\\
0 & 1 & 0 & 0 & 0 & 4 & -2 & 1 & 0 & 0 & 0 & 0 & 0 & -4 & 0 & 0 & -2 & 1 & 0\\
0 & 0 & 1 & 0 & 0 & 2 & -1 & 1 & -1 & 0 & 0 & 0 & 0 & -2 & 0 & 0 & -1 & 1 & 0\\
0 & 0 & 0 & 1 & 0 & -1 & 1 & 0 & 1 & 0 & 0 & 0 & 0 & 1 & 0 & 0 & 1 & 0 & 0\\
0 & 0 & 0 & 0 & 1 & -2 & 2 & -1 & 1 & 0 & 0 & 0 & 0 & 2 & 0 & 0 & 1 & 0 & 0\\
0 & 0 & 0 & 0 & 0 & -1 & 0 & 0 & -1 & 0 & 0 & 1 & 0 & 1 & 0 & 0 & 0 & 0 & 0\\
0 & 0 & 0 & 0 & 0 & -2 & 0 & -1 & 0 & 0 & 0 & 0 & 0 & 2 & 0 & 0 & 1 & -1 & 0\\
0 & 0 & 0 & 0 & 0 & 2 & -1 & 0 & 0 & 0 & 0 & 0 & 1 & -2 & 0 & 0 & -1 & 0 & 0\\
0 & 0 & 0 & 0 & 0 & -3 & 2 & -1 & 1 & 0 & 0 & 0 & 0 & 3 & 0 & 0 & 2 & -1 & 0\\
0 & 0 & 0 & 0 & 0 & 0 & 1 & 0 & 1 & 0 & 0 & 0 & 0 & 0 & 0 & 0 & 0 & 0 & 0\\
0 & 0 & 0 & 0 & 0 & -2 & 1 & -1 & 2 & 0 & 0 & 0 & 0 & 3 & 0 & -1 & 2 & -1 & 0\\
0 & 0 & 0 & 0 & 0 & 1 & 1 & -1 & 2 & 0 & 0 & 0 & 0 & 0 & 1 & -1 & 0 & 0 & 0\\
0 & 0 & 0 & 0 & 0 & 3 & -3 & 0 & 0 & 0 & 0 & 0 & 0 & -3 & 0 & 0 & -2 & 0 & 0\\
0 & 0 & 0 & 0 & 0 & -4 & 2 & -2 & 1 & 0 & 0 & 0 & 0 & 3 & 0 & 0 & 2 & -1 & 1\\
0 & 0 & 0 & 0 & 0 & 2 & -1 & 1 & 1 & 0 & 0 & 0 & 0 & -3 & 0 & 0 & -1 & 0 & 0
\end{array}\right),} }
\]

The automorphism $f$ of Theorem \ref{MAINthm} acts on the Néron-Severi
lattice by 
\[
\mbox{\footnotesize\ensuremath{f_{\mathcal{B}}=\left(\begin{array}{ccccccccccccccccccc}
0 & 1 & 0 & 0 & 0 & 0 & 1 & 0 & 1 & 1 & 0 & 0 & 0 & 0 & 0 & 0 & 0 & 0 & 1\\
0 & 1 & -1 & 0 & 0 & -1 & 0 & 0 & 0 & -1 & 0 & 1 & -1 & 0 & 0 & 0 & 0 & 0 & 0\\
0 & 0 & 0 & 0 & 0 & -1 & 0 & 0 & 0 & -1 & 0 & 0 & -1 & 0 & 0 & 0 & 0 & 0 & 0\\
0 & 0 & 0 & 0 & 0 & 1 & 0 & 0 & 0 & 1 & -1 & -1 & 0 & 0 & 0 & 0 & 0 & 1 & 0\\
0 & -1 & 0 & 0 & 0 & 2 & -1 & 0 & 0 & 1 & -1 & -2 & 1 & 0 & 0 & 0 & 0 & 0 & -1\\
0 & -3 & 2 & 0 & 0 & 2 & -1 & 1 & -1 & 2 & -2 & -3 & 2 & -1 & 0 & 0 & 0 & 0 & -1\\
0 & -1 & 1 & 0 & 0 & 1 & 0 & 0 & 0 & 1 & -1 & -2 & 1 & -1 & 0 & 0 & 0 & 0 & 0\\
1 & 0 & 0 & 0 & 0 & 0 & 0 & 0 & 0 & 0 & 1 & 1 & -1 & 1 & 1 & 0 & 0 & 0 & 0\\
1 & 1 & 0 & 0 & 0 & -1 & 0 & 0 & 0 & -1 & 2 & 2 & -1 & 1 & 0 & 0 & 0 & 0 & 1\\
0 & 1 & -1 & 0 & 0 & -1 & 1 & -1 & 1 & 0 & 0 & 1 & 0 & 0 & 0 & 0 & 0 & 0 & 1\\
0 & 1 & -1 & 0 & 0 & -1 & 1 & -1 & 1 & -1 & 1 & 2 & 0 & 0 & 0 & 0 & 0 & 0 & 1\\
0 & -1 & 1 & 0 & 0 & 1 & -1 & 0 & 0 & 1 & -1 & -1 & 1 & -1 & 0 & 0 & 0 & 0 & 0\\
1 & 2 & -1 & 0 & 0 & -2 & 1 & 0 & 0 & -1 & 2 & 2 & -1 & 1 & 0 & 1 & 1 & 0 & 1\\
1 & 0 & 0 & 0 & 0 & 0 & 0 & 1 & -1 & 0 & 1 & 0 & 0 & 0 & 0 & 1 & 0 & 0 & 0\\
1 & 1 & -1 & 0 & 1 & -1 & 0 & 0 & 0 & -1 & 2 & 2 & -1 & 1 & 0 & 1 & 0 & 0 & 0\\
1 & -1 & 1 & 1 & 0 & 0 & -1 & 1 & -1 & 0 & 1 & 0 & 0 & 0 & 0 & 1 & 0 & 0 & 0\\
-1 & -2 & 1 & 0 & 0 & 2 & -1 & 0 & 0 & 1 & -2 & -2 & 2 & -1 & 0 & -1 & 0 & 0 & -1\\
0 & 2 & -1 & 0 & 0 & -2 & 1 & -1 & 0 & -2 & 2 & 3 & -1 & 1 & 0 & 0 & 0 & 0 & 1\\
0 & -2 & 1 & 0 & 0 & 2 & -1 & 1 & -1 & 1 & -1 & -2 & 1 & 0 & 0 & 0 & 0 & 0 & -2
\end{array}\right).} }
\]

Let us define 
\[
S={\textstyle \sum_{j=1}^{9}}A_{i},\,\,S'={\textstyle \sum_{j=1}^{9}}A_{i}'
\]
The classes of the $\cu$-curves $(R_{i})$ defined in Section \ref{subsec:A-Hessian-model}
are 
\[
(R_{i})=2L-\tfrac{1}{3}(S+2S'+3A_{i}+3A'_{i}),
\]
where $L$ is the pullback of a line by the double cover map $X_{\l}\to\PP^{2}$.
The translation automorphism $\tau$ defined in Section \ref{subsec:A-Hessian-model}
sends $L$ to the class 
\[
L'=7L-\tfrac{4}{3}(S+2S').
\]

The three polynomials $A,B,D$ in $\QQ(\o)(t)$ of Section \ref{subsec:A-Weierstrass-equation}
are defined as follows: 
\[
\begin{array}{c}
A=(\l^{3}t^{2}+3\l^{3}-4t^{2})(\l^{3}t^{2}+3\l^{3}+(6\o+6)\l^{2}t^{2}+(-6\o-6)\l^{2}-4t^{2})\\
\cdot(\l^{3}t^{2}+3\l^{3}-6\l^{2}t^{2}+6\l^{2}-4t^{2})(\l^{3}t^{2}+3\l^{3}-6\o\l^{2}t^{2}+6\o\l^{2}-4t^{2}),
\end{array}
\]
\[
\begin{array}{c}
B=(\l^{6}t^{4}+6\l^{6}t^{2}+9\l^{6}+6\l^{5}t^{4}+12\l^{5}t^{2}-18\l^{5}-18\l^{4}t^{4}+36\l^{4}t^{2}-18\l^{4}\hfill\\
-8\l^{3}t^{4}-24\l^{3}t^{2}-24\l^{2}t^{4}+24\l^{2}t^{2}+16t^{4})\hfill\\
\cdot(\l^{6}t^{4}+6\l^{6}t^{2}+9\l^{6}+(-6\o-6)\l^{5}t^{4}+(-12\o-12)\l^{5}t^{2}+(18\o+18)\l^{5}-18\o\l^{4}t^{4}\\
+36\o\l^{4}t^{2}-18\o\l^{4}-8\l^{3}t^{4}-24\l^{3}t^{2}+(24\o+24)\l^{2}t^{4}+(-24\o-24)\l^{2}t^{2}+16t^{4})\\
\cdot(\l^{6}t^{4}+6\l^{6}t^{2}+9\l^{6}+6\o\l^{5}t^{4}+12\o\l^{5}t^{2}-18\o\l^{5}+(18\o+18)\l^{4}t^{4}\hfill\\
+(-36\o-36)\l^{4}t^{2}+(18\o+18)\l^{4}-8\l^{3}t^{4}-24\l^{3}t^{2}-24\o\l^{2}t^{4}+24\o\l^{2}t^{2}+16t^{4}),
\end{array}
\]
\[
\begin{array}{c}
D=((\l+2)t-(2\o+1)\l)((\l-2\o-2)t-(2\o+1)\l)((\l+2\o)t-(2\o+1)\l)\\
\cdot(t^{2}-1)((\l+2)t+(2\o+1)\l)((\l-2\o-2)t+(2\o+1)\l)((\l+2\o)t+(2\o+1)\l).
\end{array}
\]

\vspace{4mm}

\noindent David Kohel, Xavier Roulleau, \\
 Aix-Marseille Université, CNRS, Centrale Marseille, \\
 I2M UMR 7373, \\
 13453 Marseille, France \\
 \texttt{David.Kohel@univ-amu.fr} \\
 \texttt{Xavier.Roulleau@univ-amu.fr} \vspace{4mm}
 \\
 Alessandra Sarti, \\
 Laboratoire de Mathématiques et Applications, UMR CNRS 7348,\\
 Université de Poitiers, Téléport 2,\\
 Boulevard Marie et Pierre Curie, \\
 86962 Futuroscope Chasseneuil, France \\
 \texttt{sarti@math.univ-poitiers.fr} \\

http://www-math.sp2mi.univ-poitiers.fr/~sarti/ 
\end{document}